\documentclass[11pt]{amsart}

\usepackage{amsmath,amsthm,amssymb,amsfonts,mathtools}
\usepackage{enumitem}
\usepackage[colorlinks=true,linkcolor=blue,citecolor=green,
  urlcolor=blue]{hyperref}
\hypersetup{
  pdftitle={Twisted Diophantine Approximation I: Asymptotic Theory},
  pdfauthor={Taehyeong Kim and Vasiliy Neckrasov},
  pdfkeywords={twisted Diophantine approximation, shrinking targets,
    parametric geometry of numbers, successive minima, Hausdorff measure}
}

\setenumerate{leftmargin=*}
\setitemize{leftmargin=*}
\allowdisplaybreaks
\numberwithin{equation}{section}

\newtheorem{thm}{Theorem}[section]
\newtheorem{prop}[thm]{Proposition}
\newtheorem{lem}[thm]{Lemma}
\newtheorem{cor}[thm]{Corollary}
\newtheorem*{kurzweilthm}{Kurzweil's Theorem}
\newtheorem*{fuchskimthm}{Fuchs--Kim Theorem}
\newtheorem*{conjecture}{Conjecture}
\theoremstyle{definition}

\newtheorem*{question}{Question}
\theoremstyle{remark}
\newtheorem{rem}[thm]{Remark}

\DeclareMathOperator{\rank}{rank}
\newcommand{\Mat}{\mathrm{Mat}}
\newcommand{\spanR}{\operatorname{span}_{\mathbb R}}

\newcommand{\R}{\mathbb R}
\newcommand{\Z}{\mathbb Z}
\newcommand{\T}{\mathbb T}

\newcommand{\LL}{\Lambda}
\newcommand{\bq}{\mathbf q}
\newcommand{\bp}{\mathbf p}
\newcommand{\bb}{\mathbf b}
\newcommand{\bx}{\mathbf x}

\newcommand{\bv}{\mathbf v}
\newcommand{\bs}{\mathbf s}
\newcommand{\ba}{\mathbf a}
\newcommand{\norm}[1]{\left|#1\right|}
\newcommand{\set}[1]{\left\{#1\right\}}

\newcommand{\eps}{\varepsilon}
\newcommand{\Bad}{\operatorname{Bad}}
\DeclareRobustCommand{\Dmark}{%
  \texorpdfstring{\textbf{\textup{(D)}}}{(D)}}
\DeclareRobustCommand{\Dcond}{\eqref{eq:D}}

\renewcommand{\le}{\leqslant}
\renewcommand{\ge}{\geqslant}

\title[Twisted Diophantine Approximation I]
{Twisted Diophantine Approximation I:\ Asymptotic Theory}
\author{Taehyeong Kim}
\address{Department of Mathematics, Brandeis University,
Waltham, MA 02453, USA}
\email{taehyeongkim@brandeis.edu}
\author{Vasiliy Neckrasov}
\address{Department of Mathematics, Brandeis University,
Waltham, MA 02453, USA}
\email{vneckrasov@brandeis.edu}
\subjclass[2020]{11J83, 11J20}

\begin{document}

\begin{abstract}
We establish an exact zero--one law for twisted Diophantine approximation with an arbitrary fixed real matrix and a positive non-increasing approximation function satisfying dyadic regularity. 
The criterion is the convergence or divergence of a dyadic series whose summands are Khintchine--Groshev block volumes divided by homogeneous lattice-point counts. An equivalent formulation is using all successive minima of the associated diagonal lattice trajectory, capturing homogeneous clustering in every direction. 
We also obtain an almost-sure zero--infinity law for inhomogeneous
$\psi$-Lagrange constants, separate Hausdorff-measure criteria,
and dimension results.
Within this regularity class, our theorem recovers the Kurzweil and Fuchs--Kim criteria. At the critical exponent, this series determines whether the set of badly approximable shifts has zero or full Lebesgue measure.
\end{abstract}

\maketitle

\tableofcontents

\section{Introduction}

\subsection{Background and motivation}

Fix positive integers \(m,n\), and set \(\R_+=(0,\infty)\).  Write
\(\Mat_{m,n}\) for the real \(m\times n\) matrices and
\(\T^m=\R^m/\Z^m\).  We use the supremum
norm in every finite-dimensional vector space and write
\[
 \norm{\bx}_{\Z^m}:=\min_{\bp\in\Z^m}\norm{\bx-\bp}.
\]
For \(A\in\Mat_{m,n}\), \(\bb\in\T^m\), and a non-increasing function
\(\psi:\R_+\to\R_+\), the basic asymptotic inhomogeneous Diophantine
approximation problem asks whether
\begin{equation}\label{eq:psi-approximable}
 \norm{A\bq-\bb}_{\Z^m}<\psi(\norm{\bq})
\end{equation}
for infinitely many distinct \(\bq\in\Z^n\setminus\{0\}\).  When this
occurs, the pair \((A,\bb)\) is said to be
\emph{\(\psi\)-approximable}.

Fixing \(\bb\) and varying \(A\) in \eqref{eq:psi-approximable} gives
the classical inhomogeneous theory; \(\bb=\mathbf0\) is the homogeneous
specialization.  For fixed \(\bb\), the inhomogeneous
Khintchine--Groshev theorem says that the \(\psi\)-approximable matrices
form a null set in the convergence case and a conull set in the
divergence case of either of the following series \cite{Spr79}:
\begin{equation}\label{eq:KG-volume}
 \sum_{q=1}^{\infty}q^{n-1}\psi(q)^m,
 \quad\text{equivalently,}\quad
 \sum_{k=0}^{\infty}2^{kn}\psi(2^k)^m.
\end{equation}

This paper studies \emph{twisted Diophantine approximation}, where
\(A\) is fixed and \(\bb\) varies.  We write
\[
  W_A(\psi)
 :=\set{\bb\in\T^m:(A,\bb)\text{ is \(\psi\)-approximable}}
\]
for the set of \(\psi\)-approximable targets.  Here the arithmetic of
the orbit \(A\Z^n\) in \(\T^m\) becomes decisive.  Write \(\mathcal L^m\)
for normalized Lebesgue measure on \(\T^m\).  A matrix
\(A\in\Mat_{m,n}\) is said to be \emph{badly approximable} if
\[
 \inf_{\mathbf0\ne\bq\in\Z^n}
   \norm{\bq}^{n/m}\norm{A\bq}_{\Z^m}>0,
\]
and Kurzweil proved the following zero--one law \cite{Kur55}.

\begin{kurzweilthm}
If \(A\in\Mat_{m,n}\) is badly approximable, then, for every
non-increasing \(\psi:\R_+\to\R_+\),
\[
 \mathcal L^m( W_A(\psi))=
 \begin{cases}
  0&\displaystyle\text{if }
     \sum_{q=1}^{\infty}q^{n-1}\psi(q)^m<\infty,\\[6pt]
  1&\displaystyle\text{if }
     \sum_{q=1}^{\infty}q^{n-1}\psi(q)^m=\infty.
 \end{cases}
\]
\end{kurzweilthm}

Kurzweil also proved the converse: \(A\) is badly approximable if and
only if the preceding zero--one law holds for every non-increasing
\(\psi:\R_+\to\R_+\).

Kurzweil's theorem naturally raises the following question.

\begin{question}
If the assumption that \(A\) is badly approximable is relaxed, what
should replace the series above in the zero--one law?
\end{question}

In the one-dimensional case \(m=n=1\), this question was answered
completely by Fuchs and Kim \cite[Theorem~1.2]{FK16}.

\begin{fuchskimthm}
Let \(\alpha\) be irrational with convergents
\(p_k/q_k\), and let \(\psi:\R_+\to\R_+\) be non-increasing.  Then
\[
 \mathcal L^1( W_\alpha(\psi))=
 \begin{cases}
  0&\displaystyle\text{if }
     \sum_{k=0}^{\infty}\ \sum_{q=q_k}^{q_{k+1}-1}
     \min\{\psi(q),\norm{q_k\alpha}_{\Z}\}<\infty,\\[7pt]
  1&\displaystyle\text{if }
     \sum_{k=0}^{\infty}\ \sum_{q=q_k}^{q_{k+1}-1}
     \min\{\psi(q),\norm{q_k\alpha}_{\Z}\}=\infty.
 \end{cases}
\]
\end{fuchskimthm}

To state the higher-dimensional results, recall that a matrix
\(A\in\Mat_{m,n}\) is called \emph{singular} if, for every \(\eps>0\) and
all sufficiently large \(Q\), there is \(\bq\in\Z^n\) such that
\[
 0<\norm{\bq}\le Q,
 \qquad
 \norm{A\bq}_{\Z^m}<\eps Q^{-n/m}.
\]
Otherwise, \(A\) is called \emph{nonsingular}.
Note that in dimension one, nonsingularity is equivalent to irrationality.

For nonsingular matrices, the first-named author proved a sufficient
condition for full measure and posed the corresponding zero--one problem
\cite[Theorem~1.3 and Remark~1.5(2)]{Kim24}.  Hussain and Ward extended
this result to the weighted setting
\cite[Theorem~1]{HW24}; see Section~\ref{sec:weighted}.  Moshchevitin and
the second-named author obtained further sufficient conditions for full
and zero measure in terms of Diophantine properties of \(A^{\mathsf T}\)
\cite[Theorem~1.4(b),(c)]{MN25}.
These results go beyond the badly
approximable case, but do not give an exact necessary-and-sufficient
criterion for every fixed matrix.  In a related direction,
Hauke \cite{Hau26} proves Kurzweil-type theorems along
the primes and related sequences.
Twisted approximation on nondegenerate target manifolds is studied
in \cite{BSV26}.

Our main theorem answers this question for every fixed \(A\), provided
\(\psi\) is dyadically regular.  It gives an exact series by correcting
the target volume at each dyadic scale for \(A\)-dependent clustering.

\subsection{The main theorem}

For \(r,T>0\), define
\begin{align*}
 N_A(r,T)
 &:=\#\set{(\bp,\bq)\in\Z^m\times\Z^n:
       \norm{A\bq-\bp}\le r,\ \norm{\bq}\le T},
 \\*
 M_A(r,T)
 &:=\frac{r^mT^n}{N_A(r,T)}.
\end{align*}
The zero vector is included in \(N_A(r,T)\), so \(M_A(r,T)\) is always
defined.

\medskip
\noindent\textbf{Dyadic regularity \Dmark{}.}
Following the regularity terminology of
\cite[\S3, equation~(8)]{BDV06}, we say that \(\psi\) is
\emph{dyadically regular} if there exist \(\varrho\in(0,1)\) and
\(k_0\ge0\) such that
\begin{equation*}\tag{\textbf{D}}\label{eq:D}
 \psi(2^{k+1})\le \varrho\,\psi(2^k),
 \qquad k\ge k_0.
\end{equation*}
This condition is mild; for example, \(\psi_v(q):=q^{-v}\) satisfies it
for every \(v>0\).

Our main result is the following asymptotic zero--one law.

\begin{thm}\label{thm:main}
Let \(A\in\Mat_{m,n}\) be arbitrary, and let
\(\psi:\R_+\to\R_+\) be non-increasing and satisfy \Dcond{}.  Then
\begin{equation}\label{eq:main-mass}
 \mathcal L^m( W_A(\psi))=
 \begin{cases}
  0&\displaystyle\text{if }
     \sum_{k=0}^{\infty}M_A(\psi(2^k),2^k)<\infty,\\[6pt]
  1&\displaystyle\text{if }
     \sum_{k=0}^{\infty}M_A(\psi(2^k),2^k)=\infty.
 \end{cases}
\end{equation}
\end{thm}

\begin{rem}[Volume and clustering]\label{rem:volume-clustering}
The summand in Theorem~\ref{thm:main} has the form
\[
 M_A(\psi(2^k),2^k)
 =\frac{2^{kn}\psi(2^k)^m}
        {N_A(\psi(2^k),2^k)}
 =\frac{\text{Khintchine--Groshev block volume}}
        {\text{homogeneous clustering multiplicity}}.
\]
The numerator is the \(k\)-th term of the dyadic
Khintchine--Groshev series in \eqref{eq:KG-volume}, while the
denominator is the homogeneous clustering multiplicity.
The explicit inclusion of this factor is the principal feature distinguishing
our criterion from the earlier formulations discussed above.
\end{rem}

\begin{rem}[Choice of geometric scale]\label{rem:geometric-scale}
The base \(2\) is inessential.  For any \(\beta>1\),
Theorem~\ref{thm:main} remains valid with \(2^k\) replaced by
\(\beta^k\), provided that
\(\psi(\beta^{k+1})\le\varrho\psi(\beta^k)\) for some
\(\varrho\in(0,1)\) and all sufficiently large \(k\).
\end{rem}

To recast the criterion in the language of parametric geometry of numbers,
we express its summands in terms of the successive minima along a lattice
trajectory associated with \(A\).
Set \(d=m+n\), and define the lattice and flow
\begin{equation*} \LL_A=u_A\Z^d,\qquad
 u_A=\begin{pmatrix}I_m&A\\0&I_n\end{pmatrix},\qquad
 a_t=\begin{pmatrix}
      e^{t/m}I_m&0\\
      0&e^{-t/n}I_n
     \end{pmatrix}.
\end{equation*}
For a lattice \(\LL\subset\R^d\) and \(i=1,\ldots,d\), define the
\(i\)-th successive minimum by
\[
 \lambda_i(\LL)
 :=\inf\set{\rho>0:
   \LL\cap[-\rho,\rho]^d
   \text{ contains \(i\) linearly independent vectors}}.
\]
For \(k\ge0\), define
\begin{equation*} R_k=\bigl(2^{kn}\psi(2^k)^m\bigr)^{1/d},\qquad
 t_k=\frac{mn}{d}\log\frac{2^k}{\psi(2^k)}.
\end{equation*}

The criterion has the following successive-minima form.

\begin{cor}\label{cor:pgn-form}
Uniformly in \(k\),
\begin{equation}\label{eq:pgn-mass}
 M_A(\psi(2^k),2^k)
 \asymp
 \prod_{i=1}^{d}\min\{R_k,\lambda_i(a_{t_k}\LL_A)\}.
\end{equation}
Consequently the zero--one law in Theorem~\ref{thm:main} is determined
by the convergence or divergence of
\begin{equation*} \sum_{k=0}^{\infty}
 \prod_{i=1}^{d}\min\{R_k,\lambda_i(a_{t_k}\LL_A)\}.
\end{equation*}
\end{cor}

\begin{proof}
Equation~\eqref{eq:pgn-mass} follows directly from
Lemma~\ref{lem:minima-mass}.
\end{proof}

\begin{rem}[The badly approximable case]\label{rem:kurzweil}
Suppose that \(A\) is badly approximable and that \(\psi\) satisfies
the hypotheses of Theorem~\ref{thm:main}.  By the Dani correspondence
\cite[Theorem~2.20]{Dan85}, the forward orbit
\(\{a_t\LL_A:t\ge0\}\) is bounded in the space of unimodular lattices.
Hence all of its successive minima are bounded above and
away from zero, and
Corollary~\ref{cor:pgn-form} gives, for all sufficiently large \(k\),
\[
 M_A(\psi(2^k),2^k)
 \asymp
 \min\{1,R_k^d\}
 =\min\{1,2^{kn}\psi(2^k)^m\}.
\]
This series converges if and only if the dyadic series in
\eqref{eq:KG-volume} does.  Indeed, both diverge if
\(2^{kn}\psi(2^k)^m\ge1\) for infinitely many \(k\); otherwise their
tails agree.  Thus
Theorem~\ref{thm:main} recovers the
dyadically regular case of Kurzweil's theorem.
\end{rem}

\subsection{Applications}

Part~I develops the main theorem in three directions.

\medskip
\noindent\emph{Lagrange constants and bad targets.}
Let \(\psi:\R_+\to\R_+\) be an approximation function.
We define the \emph{inhomogeneous \(\psi\)-Lagrange constant}
of \(\bb\) with respect to \(A\) by
\begin{equation}\label{eq:affine-constant}
 \mathfrak c_{A,\psi}(\bb)
 :=\liminf_{\norm{\bq}\to\infty}
   \frac{\norm{A\bq-\bb}_{\Z^m}}{\psi(\norm{\bq})}.
\end{equation}
Corollary~\ref{cor:normalized-error} shows that this constant is almost
surely either \(0\) or \(+\infty\), according as the series in
Theorem~\ref{thm:main} diverges or converges.

For
\(\psi_{n/m}(q)=q^{-n/m}\), write
\begin{equation}\label{eq:critical-notation}
 \mathfrak c_A(\bb)
 :=\mathfrak c_{A,\psi_{n/m}}(\bb),
 \qquad
 \Bad_A:=\set{\bb:\mathfrak c_A(\bb)>0}.
\end{equation}
Set
\begin{equation}\label{eq:critical-data}
 \Delta_k(A)
 :=
 \prod_{i=1}^{d}
 \min\left\{1,\lambda_i\left(a_{nk\log2}\LL_A\right)\right\}.
\end{equation}
Proposition~\ref{prop:Bad-classification} gives the following exact
dichotomy for \(\mathcal L^m\)-almost every \(\bb\in\T^m\):
\[
 \mathfrak c_A(\bb)=
 \begin{cases}
  0&\text{if }\sum_k\Delta_k(A)=\infty,\\
  +\infty&\text{if }\sum_k\Delta_k(A)<\infty.
 \end{cases}
\]
Equivalently, \(\Bad_A\) is null in the first case and conull in the
second.  This gives an exact Lebesgue-measure
answer to the fixed-matrix characterization problem raised in
\cite[Remark~1.8(3)]{Kim24}, and to the Lebesgue-measure case of
\cite[Question~1.3]{MRS24}.

For nonsingular \(A\), the fact that \(\Bad_A\) has Lebesgue measure zero was established in
\cite{Kim07,Shapira13,Kim24,Mosh23,BDGW24,MRS24} in their respective settings;
Corollary~\ref{cor:nonsingular} recovers this conclusion in the unweighted
real setting.
For every \(A\), the set \(\Bad_A\) has full Hausdorff dimension \(m\)
\cite{BHKV10,ET11,Mosh11}.

\medskip
\noindent\emph{Hausdorff measure and dimension.}
Assuming \Dcond{}, Corollary~\ref{cor:hausdorff} gives, for \(0<s<m\),
separate sufficient conditions for
\(\mathcal H^s( W_A(\psi))=0\) and for
\(\mathcal H^s( W_A(\psi))=\infty\).  These conditions are not
complementary in general.  In dimension one, Kim--Rams--Wang proved an
exact Jarn\'ik-type dimension formula for every irrational \(\alpha\) and
every non-increasing \(\psi\) \cite{KRW18}.  See
Section~\ref{subsec:hausdorff-gap} for the higher-dimensional problem.

If \(\sum_k\Delta_k(A)=\infty\), then
Corollary~\ref{cor:power-dimension} gives, for every \(v>0\),
\[
 \dim_H W_A(\psi_v)
 =\min\left\{m,\frac nv\right\}.
\]
In particular, it holds for every nonsingular \(A\).  For \(m=n=1\),
it gives the same dimension formula as in the work of Bugeaud and
Schmeling--Troubetzkoy \cite{Bug03,ST03}.

\medskip
\noindent\emph{The Fuchs--Kim criterion.}
When \(m=n=1\) and \(A=\alpha\) is irrational,
Proposition~\ref{prop:scalar} shows that the series in
Theorem~\ref{thm:main} converges if and only if the Fuchs--Kim series
above does.  Thus Theorem~\ref{thm:main} gives the Fuchs--Kim
zero--one law for \(\psi\) satisfying \Dcond{}.

\subsection{Outline of the proof of Theorem~\ref{thm:main}}

We associate each summand in \eqref{eq:main-mass} with a union of
target balls.  Put
\[
 T_k:=2^k,\qquad r_k:=\psi(T_k),\qquad
 M_k:=M_A(r_k,T_k).
\]
For \(\bx\in\R^m\), write \(\langle\bx\rangle=\bx+\Z^m\), and let
\(B(\bb,r)\) denote the open ball of radius \(r\) about
\(\bb\in\T^m\).  At the \(k\)-th scale, consider the target balls
\[
 \set{B(\langle A\bq\rangle,r_k):
       \bq\in\Z^n,\ \norm{\bq}\le T_k}.
\]
Let \(\mathcal G_k\) be a pairwise disjoint Vitali subfamily and set
\(E_k:=\bigcup_{B\in\mathcal G_k}B\).  The lattice-counting estimates
in Lemma~\ref{lem:cumulative-mass} give, for all sufficiently large
\(k\),
\begin{equation*} \#\mathcal G_k\asymp
 \frac{T_k^n}{N_A(r_k,T_k)},
 \qquad
 \mathcal L^m(E_k)
 \asymp
 \frac{r_k^mT_k^n}{N_A(r_k,T_k)}
 =M_k.
\end{equation*}
Thus \(M_k\), the summand in Theorem~\ref{thm:main}, is comparable to
the measure of the selected union at the \(k\)-th scale;
\(N_A(r_k,T_k)\) records clustering among its centers.

For the convergence part, define
\[
 F_k:=
 \bigcup_{\substack{\bq\in\Z^n\\ \norm{\bq}\le T_{k+1}}}
 B(\langle A\bq\rangle,r_k).
\]
By monotonicity and Lemmata~\ref{lem:cumulative-mass}
and \ref{lem:dilation},
\[
  W_A(\psi)\subset\limsup_{k\to\infty}F_k,
 \qquad
 \mathcal L^m(F_k)\asymp M_k.
\]
Thus \(\sum_kM_k<\infty\) implies
\(\mathcal L^m( W_A(\psi))=0\) by the first Borel--Cantelli lemma.

Now suppose that \(\sum_kM_k=\infty\), and put
\(S_N:=\sum_{k=1}^NM_k\).  By the fixed-matrix zero--one law,
Proposition~\ref{prop:zero-one}, it suffices to prove
\(\mathcal L^m( W_A(\psi))>0\).
Following the approach of Beresnevich and Velani
\cite[Lemma~2 and Theorem~3]{BV10}, we prove the
quasi-independence-on-average estimate
\begin{equation}\label{eq:intro-QIA}
 \sum_{k,\ell=1}^N\mathcal L^m(E_k\cap E_\ell)
 \ll S_N^2+S_N.
\end{equation}
Since \(\mathcal L^m(E_k)\asymp M_k\), the divergent
Borel--Cantelli lemma gives
\[
 \mathcal L^m(\limsup_{k\to\infty}E_k)
 \ge
 \limsup_{N\to\infty}
 \frac{\bigl(\sum_{k=1}^N\mathcal L^m(E_k)\bigr)^2}
      {\sum_{k,\ell=1}^N\mathcal L^m(E_k\cap E_\ell)}
 >0.
\]
Moreover,
\[
 \left(\limsup_{k\to\infty}E_k\right)
 \setminus\{\langle A\bq\rangle:\bq\in\Z^n\}
 \subset W_A(\psi).
\]
Hence Proposition~\ref{prop:zero-one} gives
\(\mathcal L^m( W_A(\psi))=1\).

It remains to prove \eqref{eq:intro-QIA}.  For \(k<\ell\), put
\(M_{k,\ell}:=M_A(r_k,T_\ell)\).  The localized packing estimate of
Section~\ref{sec:mixed} gives
\[
 \mathcal L^m(E_k\cap E_\ell)
 \ll\frac{M_kM_\ell}{M_{k,\ell}}.
\]
Section~\ref{sec:modules} proves
\[
 \sum_{1\le k<\ell\le N}
 \frac{M_kM_\ell}{M_{k,\ell}}
 \ll S_N^2+S_N.
\]
The proof groups the pairs \((k,\ell)\) by the primitive
sublattice of \(\LL_A\) spanned by the lattice points counted by
\(N_A(r_k,T_\ell)\).  Exterior-algebra estimates give
\(M_{k,\ell}\asymp1\) when this sublattice has full rank, so those pairs
contribute \(O(S_N^2)\).  For proper sublattices, the same estimates,
together with the geometric growth of \(T_k\) and the decay in
\Dcond{}, give \(O(S_N)\).  This proves \eqref{eq:intro-QIA}.

\medskip
\noindent\textbf{Companion paper.}
The companion paper~\cite{KNU26} studies uniform twisted approximation
for arbitrary fixed matrices. Under appropriate regularity assumptions,
it gives exact Hausdorff dimension formulae for the intersection and
union of uniform approximation sets over positive multiplicative
constants. For power functions outside an explicit finite set of
possible exceptional exponents, it also determines the dimension for
every fixed positive constant. These formulae extend to continuous,
strictly decreasing functions tending to zero with the same logarithmic
order. The two papers form parts of the same project, with the same
ideas leading to complementary results in the asymptotic and uniform
settings.

\medskip
\noindent\textbf{Organization.}
Part~I develops the applications of Theorem~\ref{thm:main}.
Section~\ref{sec:critical} studies \(\psi\)-Lagrange constants and
bad targets.  Section~\ref{sec:hausdorff} proves the Hausdorff-measure
and dimension results, and Section~\ref{sec:classical} compares our
criterion with that of Fuchs and Kim.  Part~II proves
Theorem~\ref{thm:main}.  Section~\ref{sec:preliminaries} gives the
lattice-counting estimates, and Section~\ref{sec:zero-one} proves the
fixed-matrix zero--one law.  Section~\ref{sec:vitali} establishes the
Vitali estimates, proves convergence, and reduces divergence to
quasi-independence on average.  Section~\ref{sec:mixed} bounds
mixed-scale intersections, and Section~\ref{sec:modules} proves
quasi-independence on average.  Section~\ref{sec:conclusion} discusses
extensions and limitations of the method.

\medskip
\noindent\textbf{Conventions.}
For nonnegative quantities \(X\) and \(Y\), we write \(X\ll Y\) if
\(X\leq CY\) for some constant \(C>0\), write \(X\gg Y\) if \(Y\ll X\),
and write \(X\asymp Y\) if both \(X\ll Y\) and \(Y\ll X\).
Since the dimensions \(m,n\), the matrix \(A\), and the function \(\psi\)
are fixed throughout, dependence on them is suppressed in
\(\ll,\gg,\asymp\), and \(O(\cdot)\).
Dependence on any additional parameters is recorded by subscripts.

\medskip
\noindent\textbf{Use of artificial intelligence.} OpenAI Codex was used as an auxiliary tool in checking selected calculations and arguments and in refining the exposition. The authors independently verified every mathematical argument, determined the final content and wording, and take full responsibility for the manuscript.

\medskip
\noindent\textbf{Acknowledgments.}
The authors are grateful to Dmitry Kleinbock, Uri Shapira, and Nikolay
Moshchevitin for helpful discussions.

\part{Applications}

We apply Theorem~\ref{thm:main} to \(\psi\)-Lagrange constants and
Hausdorff measure and dimension, and compare our series criterion with
the Fuchs--Kim criterion in dimension one.

\section{Inhomogeneous Lagrange constants}\label{sec:critical}

This section studies \(\psi\)-Lagrange constants for a prescribed
approximation function \(\psi\).  We first prove an almost-sure
zero--infinity law in this generality and then specialize it at exponent
\(n/m\), where the criterion has a particularly simple successive-minima
form.

Recall that the \(\psi\)-Lagrange constant \(\mathfrak c_{A,\psi}\)
is defined in \eqref{eq:affine-constant}.  Theorem~\ref{thm:main}
immediately yields the following zero--infinity law.

\begin{cor}
\label{cor:normalized-error}
Under the hypotheses of Theorem~\ref{thm:main}, for
\(\mathcal L^m\)-almost every \(\bb\in\T^m\),
\[
 \mathfrak c_{A,\psi}(\bb)=
 \begin{cases}
  0&\text{if }\sum_kM_A(\psi(2^k),2^k)=\infty,\\
  +\infty&\text{if }\sum_kM_A(\psi(2^k),2^k)<\infty.
 \end{cases}
\]
\end{cor}

\begin{proof}
To apply Theorem~\ref{thm:main} to the rescaled functions
\(j^{-1}\psi\) and \(j\psi\) for every \(j\ge1\), first note that its series criterion is
invariant under fixed positive rescaling.  Indeed, \(c\psi\) satisfies
\Dcond{} with the same contraction factor, and for every fixed \(c>0\),
Lemma~\ref{lem:dilation} gives
\begin{equation*} M_A(c\psi(2^k),2^k)\asymp_cM_A(\psi(2^k),2^k)
\end{equation*}
In the divergent case,
Theorem~\ref{thm:main} gives full measure to
\( W_A(j^{-1}\psi)\) for every \(j\ge1\).  Hence their countable intersection
has full measure and is contained in
\(\{\mathfrak c_{A,\psi}=0\}\), which proves the divergent case.

In the convergent case, Theorem~\ref{thm:main} shows that
\( W_A(j\psi)\) is null for every \(j\ge1\).  Outside
their countable union, for each \(j\ge1\) the inequality
\[
 \norm{A\bq-\bb}_{\Z^m}<j\psi(\norm{\bq})
\]
has only finitely many solutions.  Thus
\(\mathfrak c_{A,\psi}(\bb)\ge j\) for every \(j\), so
\(\{\mathfrak c_{A,\psi}=+\infty\}\) has full measure, proving the
convergent case.
\end{proof}

For \(\psi_{n/m}(q)=q^{-n/m}\), recall the
notation \(\mathfrak c_A\), \(\Bad_A\), and \(\Delta_k(A)\) from
\eqref{eq:critical-notation} and \eqref{eq:critical-data}.  For
\(\eps>0\), set
\begin{align*}
 \Bad_A(\eps)
 &:=\T^m\setminus W_A(\eps\psi_{n/m}),
 \\
 \Bad_A(\infty)
 &:=\set{\bb:\mathfrak c_A(\bb)=+\infty}.
\end{align*}
By \eqref{eq:affine-constant}, these sets satisfy
\[
 \Bad_A=\bigcup_{j\ge1}\Bad_A(1/j),
 \qquad
 \Bad_A(\infty)=\bigcap_{j\ge1}\Bad_A(j).
\]
The following proposition gives the corresponding classification.

\begin{prop}
\label{prop:Bad-classification}
For every \(A\in\Mat_{m,n}\), uniformly in \(k\ge0\),
\begin{equation}\label{eq:critical-equivalence}
 M_A(\psi_{n/m}(2^k),2^k)
 =\frac{1}{N_A(\psi_{n/m}(2^k),2^k)}
 \asymp\Delta_k(A).
\end{equation}
Moreover, for \(\mathcal L^m\)-almost every \(\bb\in\T^m\),
\[
 \mathfrak c_A(\bb)=
 \begin{cases}
  0&\text{if }\sum_k\Delta_k(A)=\infty,\\
  +\infty&\text{if }\sum_k\Delta_k(A)<\infty.
 \end{cases}
\]
Equivalently, for every \(\eps>0\), the sets \(\Bad_A(\eps)\),
\(\Bad_A\), and \(\Bad_A(\infty)\) are all null in the first case and
all conull in the second.
\end{prop}

\begin{proof}
Take \(\psi=\psi_{n/m}\).  Then
\[
 \psi_{n/m}(2^{k+1})=2^{-n/m}\psi_{n/m}(2^k),
\]
so \Dcond{} holds with contraction factor \(2^{-n/m}\).  Moreover,
\(R_k=1\) and \(t_k=nk\log2\).  The definition of \(M_A\) and
Corollary~\ref{cor:pgn-form} give \eqref{eq:critical-equivalence}.
The conclusion about \(\mathfrak c_A\) follows from
Corollary~\ref{cor:normalized-error}.  If the almost-sure value is zero,
each \(\Bad_A(\eps)\), their countable union \(\Bad_A\), and
\(\Bad_A(\infty)\subset\Bad_A\) are null.  If the almost-sure value is
infinite, \(\Bad_A(\infty)\), and hence all the other displayed sets,
are conull.
\end{proof}

\begin{rem}\label{rem:critical-comparison}
\leavevmode
\begin{enumerate}[
  label=\textup{(\roman*)},
  topsep=.4\baselineskip,
  itemsep=.4\baselineskip
]
\item \emph{Lebesgue measure.}  The question in
\cite[Remark~1.8(3)]{Kim24} is answered by
Proposition~\ref{prop:Bad-classification}: for every
\(A\in\Mat_{m,n}\), the set \(\Bad_A(\eps)\) is null for some
(equivalently, every) \(\eps>0\) if and only if
\(\sum_k\Delta_k(A)=\infty\).  It also settles the Lebesgue-measure case of
\cite[Question~1.3]{MRS24}.

\item \emph{Hausdorff dimension.}  Although not needed here, it is worth
mentioning that \cite[Theorem~1.3]{KKL25} proves
\[
 A\text{ is \emph{singular on average}}
 \quad\Longleftrightarrow\quad
 \dim_H\Bad_A(\eps)=m\quad\text{for some }\eps>0.
\]
In dimension one, see
\cite[Theorem~1.1]{BKLR21}; see also \cite{LSS19}.
For results on \(\Bad_A(\infty)\), see also \cite{ET11,Kim25}.
\end{enumerate}
\end{rem}

\begin{cor}\label{cor:nonsingular}
If \(A\) is nonsingular, then
\[
 \sum_{k=1}^{\infty}\Delta_k(A)=\infty.
\]
Consequently, \(\mathcal L^m(\Bad_A)=0\), equivalently
\(\mathfrak c_A=0\) almost everywhere.
\end{cor}

\begin{proof}
By the Dani correspondence \cite[Theorem~2.14]{Dan85}, singularity of
\(A\) is equivalent to divergence of the forward orbit
\(\{a_t\LL_A:t\ge0\}\) in the space of unimodular lattices.  Since
\(A\) is nonsingular, this orbit meets some compact set at arbitrarily
large times.  By Mahler's compactness criterion, its first minimum is
bounded away from zero there.  Thus there are \(c_0>0\) and arbitrarily
large \(t\) such that
\[
 \lambda_1(a_t\LL_A)\ge c_0.
\]
Since the first minimum changes by at most a fixed multiplicative factor
over bounded time intervals, sampling at \(t=nk\log2\) gives \(c>0\)
and infinitely many \(k\) such that
\(\lambda_1(a_{nk\log2}\LL_A)\ge c\).  Since
\(\lambda_i\ge\lambda_1\) for every
\(i\), at each such index
\[
 \Delta_k(A)\ge\min\{1,c\}^d,
\]
so \(\sum_k\Delta_k(A)=\infty\), and the remaining conclusions follow
from Proposition~\ref{prop:Bad-classification}.
\end{proof}

\section{Hausdorff measure and dimension}\label{sec:hausdorff}

We give separate sufficient conditions for the \(s\)-dimensional Hausdorff
measure \(\mathcal H^s( W_A(\psi))\) to be zero or infinite and use them to
obtain dimension bounds for the approximation functions
\(\psi_v(q)=q^{-v}\).  The convergence criterion uses a direct cover at radius
\(\psi(2^k)\), whereas the divergence criterion uses the mass transference
principle \cite{BV06} with the enlarged radius
\(\psi(2^k)^{s/m}\); this explains the different clustering counts below.

\begin{cor}\label{cor:hausdorff}
Let \(A\in\Mat_{m,n}\), let \(\psi:\R_+\to\R_+\) be non-increasing
and satisfy \Dcond{}, and let \(0<s<m\).  Then
\begin{equation}\label{eq:Hausdorff-cases}
 \mathcal H^s\bigl( W_A(\psi)\bigr)=
 \begin{cases}
  0&\displaystyle\text{if }
  \sum_{k=0}^{\infty}
  \frac{2^{kn}\psi(2^k)^s}
       {N_A(\psi(2^k),2^k)}<\infty,\\[10pt]
  \infty&\displaystyle\text{if }
  \sum_{k=0}^{\infty}
  \frac{2^{kn}\psi(2^k)^s}
       {N_A(\psi(2^k)^{s/m},2^k)}=\infty.
 \end{cases}
\end{equation}
These are separate implications; the two series conditions need not be
complementary.
\end{cor}

\begin{proof}
The convergence case is proved together with that of
Theorem~\ref{thm:main} in Section~\ref{subsec:convergence-case}.

For the divergence case, put \(\varphi=\psi^{s/m}\).
If \(\varrho\) is a contraction factor for \(\psi\), then
\(\varphi(2^{k+1})\le\varrho^{s/m}\varphi(2^k)\) for all sufficiently
large \(k\).  Thus \(\varphi\) is non-increasing and satisfies
\Dcond{}, while the divergent series in \eqref{eq:Hausdorff-cases} is
\[
 \sum_k M_A(\varphi(2^k),2^k).
\]
Theorem~\ref{thm:main} therefore gives
\(\mathcal L^m( W_A(\varphi))=1\).

For each \(\mathbf0\ne\bq\in\Z^n\), put
\[
 B_{\bq}:=B(\langle A\bq\rangle,\psi(\norm{\bq})),
 \qquad r(B_{\bq}):=\psi(\norm{\bq}),
\]
where \(r(B_{\bq})\) denotes the radius of \(B_{\bq}\).
The torus with the supremum metric satisfies
\(\mathcal H^m(B(\bb,r))\asymp r^m\) for all sufficiently small \(r\).
We apply the general mass transference principle
\cite[\S6.1, Theorem~3]{BV06} with
\(X=\T^m\), \(g(r)=r^m\), and \(f(r)=r^s\).  In its notation,
\[
 B_{\bq}^m=B_{\bq},
 \qquad
 B_{\bq}^s=B(\langle A\bq\rangle,\varphi(\norm{\bq})).
\]
Since \(r(B_{\bq})\to0\), \(r\mapsto r^{s-m}\) is monotonic, and
\(\limsup_{\norm{\bq}\to\infty}B_{\bq}^s= W_A(\varphi)\) is conull,
the mass transference principle gives
\[
 \mathcal H^s( W_A(\psi))
 =\mathcal H^s\!\left(\limsup_{\norm{\bq}\to\infty}B_{\bq}^m\right)
 =\infty.
\]
\end{proof}

\begin{cor}
\label{cor:power-dimension}
Let \(A\in\Mat_{m,n}\), and suppose that
\[
 \sum_{k=1}^{\infty}\Delta_k(A)=\infty.
\]
Then, for every \(v>0\),
\begin{equation}\label{eq:power-dimension}
 \dim_H W_A(\psi_v)
 =\min\left\{m,\frac nv\right\}.
\end{equation}
In particular, \eqref{eq:power-dimension} holds for every nonsingular
matrix \(A\).
\end{cor}

\begin{proof}
For every \(v>0\), the function \(\psi_v\) is non-increasing and
\[
 \psi_v(2^{k+1})=2^{-v}\psi_v(2^k),
\]
so it satisfies \Dcond{}.

Suppose first that \(0<v\le n/m\).  Since
\(\psi_v(q)\ge\psi_{n/m}(q)\) for \(q\ge1\),
\[
  W_A(\psi_{n/m})\subseteq W_A(\psi_v),
\]
while Proposition~\ref{prop:Bad-classification} shows that
\( W_A(\psi_{n/m})\) is conull.  It follows that
\( W_A(\psi_v)\) is conull and has Hausdorff dimension \(m\), as
asserted in \eqref{eq:power-dimension}.

Now suppose that \(v>n/m\), and put \(s=n/v\), so \(0<s<m\).  Since
\[
 2^{kn}\psi_v(2^k)^s=1,
 \qquad
 \psi_v(2^k)^{s/m}=\psi_{n/m}(2^k),
\]
we have, by \eqref{eq:critical-equivalence},
\[
 \sum_{k=0}^{\infty}
 \frac{2^{kn}\psi_v(2^k)^s}
      {N_A(\psi_v(2^k)^{s/m},2^k)}
 =\sum_{k=0}^{\infty}
 \frac{1}{N_A(\psi_{n/m}(2^k),2^k)}
 \asymp\sum_{k=0}^{\infty}\Delta_k(A)=\infty.
\]
The divergence part of Corollary~\ref{cor:hausdorff} therefore gives
\(\mathcal H^{n/v}( W_A(\psi_v))=\infty\), and hence
\(\dim_H W_A(\psi_v)\ge n/v\).

For the reverse inequality, fix \(s'\) with \(n/v<s'<m\).  The zero
vector counted in \(N_A\) gives
\(N_A(\psi_v(2^k),2^k)\ge1\), while \(n-vs'<0\).  Hence
\[
 \sum_k
 \frac{2^{k(n-vs')}}{N_A(\psi_v(2^k),2^k)}
 \le\sum_k2^{k(n-vs')}<\infty,
\]
and the convergence part of Corollary~\ref{cor:hausdorff} gives
\(\mathcal H^{s'}( W_A(\psi_v))=0\).  Since this holds for every
\(s'>n/v\), it follows that \(\dim_H W_A(\psi_v)\le n/v\).
The last assertion follows from Corollary~\ref{cor:nonsingular}.
\end{proof}

For \(m=n=1\), this gives the same dimension value as in the work of
Bugeaud and Schmeling--Troubetzkoy \cite{Bug03,ST03}.

\section{The Fuchs--Kim criterion}\label{sec:classical}

We now specialize the criterion to an irrational number \(\alpha\).
Let \(q_k\) denote the denominators of its continued-fraction
convergents.

\begin{prop}\label{prop:scalar}
Let \(\psi:\R_+\to\R_+\) be non-increasing and satisfy \Dcond{}.
\begin{equation}\label{eq:scalar-equivalence}
 \sum_{j=0}^{\infty}M_\alpha(\psi(2^j),2^j)<\infty
 \quad\Longleftrightarrow\quad
 \sum_{k=0}^{\infty}\ \sum_{q=q_k}^{q_{k+1}-1}
 \min\{\psi(q),\norm{q_k\alpha}_{\Z}\}<\infty.
\end{equation}
\end{prop}

\begin{proof}
For \(q_s\le q<q_{s+1}\), put
\[
 h(q):=\min\{\psi(q),\norm{q_s\alpha}_{\Z}\}.
\]
The function \(h\) is non-increasing, so Cauchy condensation gives
\[
 \sum_qh(q)<\infty
 \quad\Longleftrightarrow\quad
 \sum_j2^jh(2^j)<\infty.
\]
For all large \(j\), we have \(\psi(2^j)<1/2\);
choose \(s\) with \(q_s\le2^j<q_{s+1}\).

Suppose first that the series on the left of
\eqref{eq:scalar-equivalence} converges.
The best approximation property of the convergent denominator \(q_s\)
gives \(\norm{q\alpha}_{\Z}\ge\norm{q_s\alpha}_{\Z}\) for
\(1\le q<q_{s+1}\) \cite[\S2.2]{FK16}.
Hence the points \(q\alpha-p\in[-\psi(2^j),\psi(2^j)]\), where
\(q,p\in\Z\) and \(0\le q\le2^j\), are
\(\norm{q_s\alpha}_{\Z}\)-separated.
Counting them and restoring both signs gives
\[
\begin{aligned}
 N_\alpha(\psi(2^j),2^j)
 &\le1+\frac{4\psi(2^j)}{\norm{q_s\alpha}_{\Z}}
 \le\frac{5\psi(2^j)}{h(2^j)},\\
 M_\alpha(\psi(2^j),2^j)&\ge\frac15 2^jh(2^j).
\end{aligned}
\]
The right-hand series therefore converges by condensation.

Conversely, suppose that \(\sum_qh(q)<\infty\).
The estimate in \cite[proof of Lemma~2.1]{FK16} implies
\(\sum_sq_s\psi(q_s)<\infty\).
By the homogeneous counting bound \cite[\S6.1, (6.9)]{BHV20},
adjoining the zero vector and using symmetry gives
\[
\begin{aligned}
 N_\alpha(\psi(2^j),2^j)
 &\ge1+2\left\lfloor
       \min\left\{\psi(2^j)q_{s+1},\frac{2^j}{q_s}\right\}
       \right\rfloor\\
 &\gg\max\left\{1,
       \min\left\{\frac{\psi(2^j)}{\norm{q_s\alpha}_{\Z}},
                   \frac{2^j}{q_s}\right\}
       \right\},
\end{aligned}
\]
where we used \(q_{s+1}\norm{q_s\alpha}_{\Z}\ge1/2\).
If \(\psi(2^j)\le\norm{q_s\alpha}_{\Z}\), then
\(h(2^j)=\psi(2^j)\), and \(N_\alpha(\psi(2^j),2^j)\ge1\) gives
\[
 M_\alpha(\psi(2^j),2^j)
 =\frac{2^j\psi(2^j)}{N_\alpha(\psi(2^j),2^j)}
 \le2^j\psi(2^j)=2^jh(2^j).
\]
If \(\psi(2^j)>\norm{q_s\alpha}_{\Z}\), then
\(h(2^j)=\norm{q_s\alpha}_{\Z}\), and the preceding counting bound gives
\[
\begin{aligned}
 M_\alpha(\psi(2^j),2^j)
 &\ll\frac{2^j\psi(2^j)}
 {\min\left\{\dfrac{\psi(2^j)}{\norm{q_s\alpha}_{\Z}},
                    \dfrac{2^j}{q_s}\right\}}\\
 &=\max\{2^jh(2^j),q_s\psi(2^j)\}.
\end{aligned}
\]
Thus, in both cases,
\[
 M_\alpha(\psi(2^j),2^j)\ll2^jh(2^j)+q_s\psi(2^j).
\]
Finally, let \(j_s=\lceil\log_2q_s\rceil\), so that \(2^{j_s}\ge q_s\).
For all sufficiently large \(s\), iterating \Dcond{} and using
monotonicity gives
\[
 \psi(2^{j_s+\ell})\le\varrho^\ell\psi(2^{j_s})
 \le\varrho^\ell\psi(q_s)\qquad(\ell\ge0).
\]
Since every \(j\) with \(q_s\le2^j<q_{s+1}\) satisfies \(j\ge j_s\),
we obtain
\[
 \sum_{\{j:q_s\le2^j<q_{s+1}\}}q_s\psi(2^j)
 \le q_s\psi(q_s)\sum_{\ell=0}^{\infty}\varrho^\ell
 =\frac{q_s\psi(q_s)}{1-\varrho}.
\]
Summing over \(s\), and using the convergence of
\(\sum_sq_s\psi(q_s)\) and \(\sum_j2^jh(2^j)\), proves convergence
of the left-hand series.
\end{proof}

\part{Proof of the asymptotic zero--one law}

We now prove Theorem~\ref{thm:main}, using lattice counts to control
the measures and intersections of dyadic target sets.
The convergence case follows from
covering estimates and the Borel--Cantelli lemma.  In the divergence
case, quasi-independence on average gives positive measure, and the
fixed-matrix zero--one law then yields full measure.

\section{Lattice-counting preliminaries}
\label{sec:preliminaries}

We collect the lattice-counting estimates used throughout Part~II.
They relate \(M_A(r,T)\) to successive minima and
show that fixed dilations change it only by constant factors.

Recall that \(\langle\bx\rangle=\bx+\Z^m\), and that \(B(\bb,r)\)
denotes the open ball of radius \(r\) about \(\bb\in\T^m\).
Throughout Part~II, we retain the notation
\[
 T_k=2^k,\qquad r_k=\psi(T_k).
\]
After modifying \(\psi\) on a bounded initial interval, we may assume
that \Dcond{} holds for all \(k\ge0\). Thus
\begin{equation}\label{eq:radius-decay}
 r_\ell\le \varrho^{\ell-k}r_k\qquad(k<\ell),
\end{equation}
and \(r_k\to0\).

For \(r,T>0\), write
\[
 \mathcal B(r,T):=[-r,r]^m\times[-T,T]^n.
\]
Thus
\[
 N_A(r,T)=\#\bigl(\LL_A\cap\mathcal B(r,T)\bigr).
\]
We recall two standard lattice-point counting estimates.

\begin{lem}\label{lem:basic-lattice-counting}
Let \(\LL\subset\R^d\) be a full-rank lattice.  For every centrally
symmetric convex body \(K\subset\R^d\) centered at the origin and \(C\ge1\),
\begin{equation}\label{eq:lattice-doubling}
 \#(\LL\cap CK)
 \ll_C\#(\LL\cap K).
\end{equation}
Moreover, for every \(R>0\),
\begin{equation}\label{eq:standard-lattice-count}
 \#(\LL\cap[-R,R]^d)
 \asymp\prod_{i=1}^{d}\max\{1,R/\lambda_i(\LL)\}.
\end{equation}
\end{lem}

\begin{proof}
Put \(a=\lceil2C\rceil\).  Applying \cite[Lemma~2.1]{Henk02} to
\(CK\) and the sublattice \(a\LL\) gives
\[
 \#(\LL\cap CK)
 \le a^d\#\bigl(\LL\cap(2C/a)K\bigr)
 \le a^d\#(\LL\cap K),
\]
which proves \eqref{eq:lattice-doubling}.

For the lower bound in \eqref{eq:standard-lattice-count}, choose
linearly independent \(\bv_i\in\LL\) with
\(\norm{\bv_i}\le\lambda_i(\LL)\).
The distinct vectors \(\sum_i a_i\bv_i\), where \(a_i\in\Z\) and
\(\lvert a_i\rvert\le R/(d\lambda_i(\LL))\), lie in \([-R,R]^d\).
Counting them gives
\[
 \#(\LL\cap[-R,R]^d)
 \ge\prod_{i=1}^{d}
 \left(2\left\lfloor\frac{R}{d\lambda_i(\LL)}\right\rfloor+1\right)
 \gg\prod_{i=1}^{d}\max\{1,R/\lambda_i(\LL)\}.
\]
The upper bound follows from \cite[Theorem~1.5]{Henk02} applied to
\([-R,R]^d\), whose successive minima are \(\lambda_i(\LL)/R\).
\end{proof}

We now express \(N_A(r,T)\) and \(M_A(r,T)\) in terms of the
successive minima of \(a_t\LL_A\).

\begin{lem}\label{lem:minima-mass}
For \(r,T>0\), put
\[
 R=(T^nr^m)^{1/d},\qquad
 t=\frac{mn}{d}\log(T/r).
\]
Then
\begin{align}
 N_A(r,T)
 &\asymp\prod_{i=1}^{d}
    \max\{1,R/\lambda_i(a_t\LL_A)\},
 \label{eq:minima-count}\\
 M_A(r,T)
 &\asymp\prod_{i=1}^{d}
    \min\{R,\lambda_i(a_t\LL_A)\}.
 \label{eq:minima-mass}
\end{align}
\end{lem}

\begin{proof}
The definitions of \(R\) and \(t\) give
\(re^{t/m}=Te^{-t/n}=R\), so \(a_t\) maps
\(\mathcal B(r,T)\) onto \([-R,R]^d\).
Applying \eqref{eq:standard-lattice-count} to \(a_t\LL_A\), we obtain
\[
\begin{aligned}
 N_A(r,T)
 &=\#\bigl(a_t\LL_A\cap[-R,R]^d\bigr)\\
 &\asymp\prod_{i=1}^{d}\max\{1,R/\lambda_i(a_t\LL_A)\},
\end{aligned}
\]
which proves \eqref{eq:minima-count}.
Since \(r^mT^n=R^d\), the definition of \(M_A\) then gives
\[
\begin{aligned}
 M_A(r,T)
 &=\frac{R^d}{N_A(r,T)}
 \asymp\prod_{i=1}^{d}
 \frac{R}{\max\{1,R/\lambda_i(a_t\LL_A)\}}\\
 &=\prod_{i=1}^{d}\min\{R,\lambda_i(a_t\LL_A)\},
\end{aligned}
\]
proving \eqref{eq:minima-mass}.
\end{proof}

We shall also use the following robustness under fixed dilations.

\begin{lem}\label{lem:dilation}
For fixed \(c_1,c_2>0\),
\begin{equation*} N_A(c_1r,c_2T)\asymp_{c_1,c_2}N_A(r,T),\qquad
 M_A(c_1r,c_2T)\asymp_{c_1,c_2}M_A(r,T).
\end{equation*}
\end{lem}

\begin{proof}
Set \(C=\max\{c_1,c_1^{-1},c_2,c_2^{-1}\}\).  Then
\[
 \mathcal B(c_1r,c_2T)\subset C\mathcal B(r,T),
 \qquad
 \mathcal B(r,T)\subset C\mathcal B(c_1r,c_2T).
\]
The first inclusion and \eqref{eq:lattice-doubling} give
\[
 N_A(c_1r,c_2T)
 \le\#\bigl(\LL_A\cap C\mathcal B(r,T)\bigr)
 \ll_C N_A(r,T).
\]
The second inclusion gives the reverse comparison, proving the
assertion for \(N_A\).  The assertion for \(M_A\) follows from
\[
 \frac{M_A(c_1r,c_2T)}{M_A(r,T)}
 =c_1^mc_2^n\frac{N_A(r,T)}{N_A(c_1r,c_2T)}
 \asymp_{c_1,c_2}1.\qedhere
\]
\end{proof}

\section{The fixed-matrix zero--one law}
\label{sec:zero-one}

The divergent argument will first produce a limsup set of positive
measure.  This section proves the fixed-matrix zero--one law that
upgrades such a conclusion to full measure.

We first record the standard Fourier consequence of one-sided invariance
under generators of a dense translation subgroup.

\begin{lem}\label{lem:dense-translations}
Let \(\bv_1,\ldots,\bv_r\in\T^m\), and suppose that
\(\Gamma:=\sum_{j=1}^r\Z\bv_j\) is dense in \(\T^m\).  If a measurable set
\(E\subset\T^m\) satisfies
\[
 E-\bv_j\subset E\qquad(1\le j\le r),
\]
then \(E\) is null or conull.
\end{lem}

\begin{proof}
By translation invariance of Lebesgue measure, each inclusion is an equality
modulo null sets.  The indicator \(f:=\mathbf1_E\) satisfies
\(f(\,\cdot+\bv_j)=f\) almost everywhere for every \(j\).  The Fourier
coefficients of \(f\) are
\[
 \widehat f(\mathbf k)
 :=\int_{\T^m}f(\bx)e^{-2\pi i\mathbf k\cdot\bx}
 \,d\mathcal L^m(\bx)\qquad(\mathbf k\in\Z^m).
\]
For every \(j\), a change of variables gives
\[
 \bigl(1-e^{2\pi i\mathbf k\cdot\bv_j}\bigr)
 \widehat f(\mathbf k)=0.
\]
Fix \(\mathbf k\ne0\).  If \(\widehat f(\mathbf k)\ne0\), the identity
above would force
\(e^{2\pi i\mathbf k\cdot\bv_j}=1\) for every \(j\), and hence
\(e^{2\pi i\mathbf k\cdot\gamma}=1\) for every \(\gamma\in\Gamma\).
Since \(\Gamma\) is dense and the character
\(\bx\mapsto e^{2\pi i\mathbf k\cdot\bx}\) is continuous, it would follow
that this character is identically \(1\) on \(\T^m\), contradicting
\(\mathbf k\ne0\).  Hence \(\widehat f(\mathbf k)=0\) for every
\(\mathbf k\ne0\).  Uniqueness of Fourier series gives
\(f=\widehat f(\mathbf0)\) almost everywhere.  Since \(f\) is an indicator
function, \(E\) is null or conull.
\end{proof}

We now deduce the fixed-matrix zero--one law.

\begin{prop}\label{prop:zero-one}
Let \(A\in\Mat_{m,n}\), and let \(\psi(N)\to0\) be non-increasing.
Then
\[
 \mathcal L^m( W_A(\psi))\in\{0,1\}.
\]
\end{prop}

\begin{proof}
We first decompose according to which coordinates escape to infinity.
Write \([n]=\{1,\ldots,n\}\).
For \(\varnothing\ne I\subset[n]\),
\(\boldsymbol{\varepsilon}=(\varepsilon_i)_{i\in I}\in\{\pm1\}^{I}\), and
\(\mathbf c=(c_j)_{j\notin I}\in\Z^{I^c}\), set
\[
 \mathcal Q_R(I,\boldsymbol{\varepsilon},\mathbf c)
 :=\set{\bq\in\Z^n:
          q_j=c_j\ (j\notin I),\ \varepsilon_iq_i\ge R\ (i\in I)}
\]
For this set, define
\[
 E_{I,\boldsymbol{\varepsilon},\mathbf c}
 :=\bigcap_{R=1}^{\infty}
   \bigcup_{\bq\in\mathcal Q_R(I,\boldsymbol{\varepsilon},\mathbf c)}
   B(\langle A\bq\rangle,\psi(\norm{\bq})).
\]
\smallskip
\noindent\textbf{Claim 1.}
We have the decomposition
\begin{equation}\label{eq:claim-one-decomposition}
  W_A(\psi)=
 \bigcup_{\varnothing\ne I\subset[n]}
 \ \bigcup_{\boldsymbol{\varepsilon}\in\{\pm1\}^I}
 \ \bigcup_{\mathbf c\in\Z^{I^c}}
 E_{I,\boldsymbol{\varepsilon},\mathbf c}.
\end{equation}
\begin{proof}[Proof of Claim 1]
To prove the inclusion \(\subseteq\), suppose that
\(\bb\in W_A(\psi)\), and choose pairwise distinct solutions
\(\bq^{(s)}\) of \eqref{eq:psi-approximable}, so that
\[
 \bb\in B(\langle A\bq^{(s)}\rangle,\psi(\lvert\bq^{(s)}\rvert))
 \qquad(s\ge1).
\]
Successively passing to subsequences in
each coordinate, we may arrange that every \(q_i^{(s)}\) is either constant
or tends to infinity with a fixed sign.  Let \(I\) be the set of coordinates
of the latter type, let \(\boldsymbol{\varepsilon}\) record their signs, and let
\(\mathbf c\) record the constant coordinates.  At least one coordinate is of
the latter type, since otherwise \(\bq^{(s)}\) would be constant in \(s\).
Hence \(I\ne\varnothing\).  For every \(R\), all
sufficiently large \(s\) satisfy
\(\bq^{(s)}\in\mathcal Q_R(I,\boldsymbol{\varepsilon},\mathbf c)\); hence
\(\bb\in E_{I,\boldsymbol{\varepsilon},\mathbf c}\).

For the inclusion \(\supseteq\), suppose that
\(\bb\in E_{I,\boldsymbol{\varepsilon},\mathbf c}\).  Then for every \(R\) there is
\(\bq_R\in\mathcal Q_R(I,\boldsymbol{\varepsilon},\mathbf c)\) such that
\[
 \bb\in B(\langle A\bq_R\rangle,\psi(\norm{\bq_R})).
\]
Since
\(I\ne\varnothing\), we have \(\norm{\bq_R}\ge R\).  Thus the collection
\(\{\bq_R:R\ge1\}\) is infinite and consists of nonzero vectors, so
\(\bb\in W_A(\psi)\).
\end{proof}

The union in \eqref{eq:claim-one-decomposition} is countable.  Hence, to prove
the proposition, it suffices to establish the following claim.

\smallskip
\noindent\textbf{Claim 2.}
For each triple \((I,\boldsymbol{\varepsilon},\mathbf c)\) as above, the set
\(E_{I,\boldsymbol{\varepsilon},\mathbf c}\) is null or conull.

\begin{proof}[Proof of Claim 2]
Fix a triple \((I,\boldsymbol{\varepsilon},\mathbf c)\), and abbreviate the
corresponding sets by \(\mathcal Q_R\) and \(E\).  Put
\[
 \Gamma_I:=\sum_{i\in I}\Z\langle A\mathbf e_i\rangle,
 \qquad H_I:=\overline{\Gamma_I},
 \qquad
 \bx_{\mathbf c}:=\left\langle
     \sum_{j\notin I}c_jA\mathbf e_j\right\rangle.
\]
For every \(\bq\in\mathcal Q_R\), we have
\[
 \langle A\bq\rangle
 =\bx_{\mathbf c}+\sum_{i\in I}q_i\langle A\mathbf e_i\rangle
 \in\bx_{\mathbf c}+\Gamma_I.
\]
Moreover, if \(\bb\in E\), then for every \(R\) there is a
\(\bq_R\in\mathcal Q_R\) such that
\[
 \operatorname{dist}(\bb,\bx_{\mathbf c}+H_I)
 \le \psi(\norm{\bq_R})\le\psi(R).
\]
Since \(\psi(R)\to0\), we obtain
\(E\subset\bx_{\mathbf c}+H_I\).  If \(H_I\ne\T^m\), this proper closed
coset has Lebesgue measure zero, and hence \(E\) is null.

It remains to consider \(H_I=\T^m\).  Fix \(i\in I\).  For
\(\bq\in\mathcal Q_{R+1}\), put
\[
 \bq'=\bq-\varepsilon_i\mathbf e_i.
\]
Then \(\bq'\in\mathcal Q_R\) and
\(\norm{\bq'}\le\norm{\bq}\).  With
\(\bv_i:=\varepsilon_i\langle A\mathbf e_i\rangle\), translation invariance of
the metric and monotonicity of \(\psi\) give
\begin{align*}
 B(\langle A\bq\rangle,\psi(\norm{\bq}))-\bv_i
 &=B(\langle A\bq'\rangle,\psi(\norm{\bq}))\\
 &\subset B(\langle A\bq'\rangle,\psi(\norm{\bq'})).
\end{align*}
Taking the union over \(\bq\in\mathcal Q_{R+1}\), and then the intersection
over \(R\), gives
\[
 E-\bv_i\subset E.
\]
Since \(i\in I\) was arbitrary and
\(\sum_{i\in I}\Z\bv_i=\Gamma_I\) is dense in \(\T^m\),
Lemma~\ref{lem:dense-translations} shows that \(E\) is null or conull.
\end{proof}

Claims~1 and~2 complete the proof.
\end{proof}

\section{Vitali covering estimates and proof reduction}
\label{sec:vitali}

We now begin the proof of Theorem~\ref{thm:main}.  We first establish the
Vitali covering estimates used in both the convergence and divergence
arguments.

For \(r>0\) and \(T\ge1\), consider the family
\begin{equation*} \mathcal P_A(r,T):=
 \bigl\{B(\langle A\bq\rangle,r):
 \bq\in[-T,T]^n\cap\Z^n\bigr\}.
\end{equation*}
By the Vitali covering lemma, we may choose a pairwise disjoint
subfamily
\[
 \mathcal G_A(r,T)\subset\mathcal P_A(r,T)
\]
such that
\begin{equation}\label{eq:vitali-cover}
 \bigcup_{B\in\mathcal G_A(r,T)}B
 \subset\bigcup_{B\in\mathcal P_A(r,T)}B
 \subset\bigcup_{B\in\mathcal G_A(r,T)}3B,
\end{equation}
where \(3B\) has the same center as \(B\) and three times its radius.

We shall use the following estimates.

\begin{lem}\label{lem:cumulative-mass}
Uniformly in \(A\in\Mat_{m,n}\), \(0<r<1/4\), and \(T\ge1\),
\begin{equation}\label{eq:vitali-cardinality}
 \#\mathcal G_A(r,T)\asymp\frac{T^n}{N_A(r,T)},
\end{equation}
and
\begin{equation*} \mathcal L^m\!\left(\bigcup_{B\in\mathcal G_A(r,T)}B\right)
 \asymp
 \mathcal L^m\!\left(\bigcup_{B\in\mathcal P_A(r,T)}B\right)
 \asymp M_A(r,T).
\end{equation*}
\end{lem}

\begin{proof}
We first prove \eqref{eq:vitali-cardinality}.  For the lower bound, the
second inclusion in \eqref{eq:vitali-cover} shows that, for every
\(\bq\in[-T,T]^n\cap\Z^n\), there is \(B\in\mathcal G_A(r,T)\) such
that \(\langle A\bq\rangle\in3B\).  Hence
\[
 \#\bigl([-T,T]^n\cap\Z^n\bigr)
 \le\sum_{B\in\mathcal G_A(r,T)}
 \#\set{\bq\in[-T,T]^n\cap\Z^n:
          \langle A\bq\rangle\in3B}.
\]
Fix \(B\in\mathcal G_A(r,T)\), and write
\[
 B=B(\langle A\bs\rangle,r),
 \qquad \bs\in[-T,T]^n\cap\Z^n.
\]
If \(\langle A\bq\rangle\in3B\), then there is \(\bp\in\Z^m\) such
that
\[
 u_A\binom{-\bp}{\bq-\bs}
 =\binom{A(\bq-\bs)-\bp}{\bq-\bs}
 \in\LL_A\cap\mathcal B(3r,2T).
\]
Distinct \(\bq\)'s give distinct second coordinates
\(\bq-\bs\).  Hence
\[
 \#\set{\bq\in[-T,T]^n\cap\Z^n:\langle A\bq\rangle\in3B}
 \le \#\bigl(\LL_A\cap\mathcal B(3r,2T)\bigr)
 =N_A(3r,2T).
\]
Lemma~\ref{lem:dilation} therefore gives
\[
 T^n\asymp\#\bigl([-T,T]^n\cap\Z^n\bigr)
 \le\#\mathcal G_A(r,T)\,N_A(3r,2T)
 \asymp\#\mathcal G_A(r,T)\,N_A(r,T),
\]
which is the lower bound in \eqref{eq:vitali-cardinality}.

For the upper bound, for each \(B\in\mathcal G_A(r,T)\), choose
\(\bs_B\in[-T,T]^n\cap\Z^n\) and \(\bp_B\in\Z^m\) such that
\[
 B=B(\langle A\bs_B\rangle,r),
 \qquad
 A\bs_B-\bp_B\in[0,1)^m.
\]
Then
\[
 u_A\binom{-\bp_B}{\bs_B}
 =\binom{A\bs_B-\bp_B}{\bs_B}\in\LL_A.
\]
Translation by this lattice point therefore gives
\[
 \#\left(
 \LL_A\cap
 \left[
 \binom{A\bs_B-\bp_B}{\bs_B}+\mathcal B(r/2,T)
 \right]\right)
 =N_A(r/2,T).
\]
The first-coordinate projection of the translated box is
\[
 A\bs_B-\bp_B+[-r/2,r/2]^m.
\]
Its image in \(\T^m\) is contained in \(B\).  Hence these translated
boxes are pairwise disjoint as \(B\) ranges over
\(\mathcal G_A(r,T)\).  Moreover, since
\(A\bs_B-\bp_B\in[0,1)^m\), \(\norm{\bs_B}\le T\), and \(r<1/4\),
\[
 \binom{A\bs_B-\bp_B}{\bs_B}+\mathcal B(r/2,T)
 \subset[-2,2]^m\times[-2T,2T]^n.
\]
Consequently,
\[
\begin{aligned}
 &\bigsqcup_{B\in\mathcal G_A(r,T)}
 \left(
 \LL_A\cap
 \left[
 \binom{A\bs_B-\bp_B}{\bs_B}+\mathcal B(r/2,T)
 \right]\right)\\
 &\qquad\subset
 \LL_A\cap\bigl([-2,2]^m\times[-2T,2T]^n\bigr).
\end{aligned}
\]
Every point of the set on the right has the form
\((A\bq+\bp,\bq)\).  There are \(O(T^n)\) choices for
\(\bq\in[-2T,2T]^n\cap\Z^n\), and, for fixed \(\bq\), only
\(O(1)\) choices of \(\bp\in\Z^m\) for which
\(A\bq+\bp\in[-2,2]^m\).  Taking cardinalities therefore gives
\[
 \#\mathcal G_A(r,T)\,N_A(r/2,T)\ll T^n.
\]
Together with Lemma~\ref{lem:dilation}, this proves the upper bound in
\eqref{eq:vitali-cardinality}.

The selected balls are disjoint, so
\[
 \mathcal L^m\!\left(\bigcup_{B\in\mathcal G_A(r,T)}B\right)
 =\#\mathcal G_A(r,T)\,(2r)^m\asymp M_A(r,T).
\]
Together with the two inclusions in \eqref{eq:vitali-cover}, this also
gives
\[
 \mathcal L^m\!\left(\bigcup_{B\in\mathcal P_A(r,T)}B\right)
 \asymp M_A(r,T).
\]
\end{proof}

\subsection{Proof of the convergence cases}
\label{subsec:convergence-case}

\begin{proof}[Proof of the convergence parts of
Theorem~\ref{thm:main} and Corollary~\ref{cor:hausdorff}]
\leavevmode\par\noindent
For any
\(\bb\in W_A(\psi)\), the set of solutions \(\bq\) to
\eqref{eq:psi-approximable} is unbounded.  If
\(T_k\le\norm{\bq}<T_{k+1}\), then monotonicity gives
\[
 \norm{A\bq-\bb}_{\Z^m}<\psi(\norm{\bq})\le r_k.
\]
Therefore
\[
  W_A(\psi)
 \subset
 \limsup_{k\to\infty}
 \bigcup_{B\in\mathcal P_A(r_k,T_{k+1})}B
 \subset
 \limsup_{k\to\infty}
 \bigcup_{B\in\mathcal G_A(r_k,T_{k+1})}3B.
\]
For all sufficiently large \(k\), Lemmata~\ref{lem:cumulative-mass}
and~\ref{lem:dilation} give
\begin{align*}
 \mathcal L^m\!\left(
  \bigcup_{B\in\mathcal P_A(r_k,T_{k+1})}B
 \right)
 &\asymp M_A(r_k,T_k),\\
 \#\mathcal G_A(r_k,T_{k+1})
 &\asymp \frac{T_k^n}{N_A(r_k,T_k)}.
\end{align*}
If \(\sum_kM_A(r_k,T_k)<\infty\), the first estimate and the first
Borel--Cantelli lemma give
\(\mathcal L^m( W_A(\psi))=0\), proving the convergence part of
Theorem~\ref{thm:main}.  If the first series in
\eqref{eq:Hausdorff-cases} converges, the second estimate gives
\[
 \sum_k\#\mathcal G_A(r_k,T_{k+1})(6r_k)^s
 \ll\sum_k\frac{T_k^nr_k^s}{N_A(r_k,T_k)}<\infty.
\]
Since the covering balls \(3B\) have diameter at most \(6r_k\to0\),
the Hausdorff--Cantelli lemma \cite[Lemma~3.10]{BD99} gives
\(\mathcal H^s( W_A(\psi))=0\), proving the convergence part of
Corollary~\ref{cor:hausdorff}.
\end{proof}

\subsection{Proof of Theorem~\ref{thm:main}: the divergence case}

For each \(k\), consider the set
\[
 E_k:=\bigcup_{B\in\mathcal G_A(r_k,T_k)}B.
\]

We follow the classical quasi-independence-on-average strategy.  The
following lemma provides a lower bound for the measure of a limsup set
and is a convenient form of the divergent Borel--Cantelli lemma; see
\cite[Lemma~5]{Spr79}.

\begin{lem}\label{lem:generalized-BC}
Let \((F_k)\) be measurable subsets of a probability space and suppose
that \(\sum_k\mu(F_k)=\infty\).  Then
\begin{equation*} \mu(\limsup_{k\to\infty}F_k)
 \ge
 \limsup_{N\to\infty}
 \frac{\bigl(\sum_{k=1}^{N}\mu(F_k)\bigr)^2}
      {\sum_{k,\ell=1}^{N}\mu(F_k\cap F_\ell)}.
\end{equation*}
\end{lem}

We will use the following quasi-independence-on-average
estimate for the events \(E_k\).

\begin{prop}
\label{prop:vitali-QIA}
For every \(N\ge1\),
\begin{equation}\label{eq:vitali-QIA}
 \sum_{k,\ell=1}^{N}\mathcal L^m(E_k\cap E_\ell)
 \ll
 \left(\sum_{k=1}^{N}M_A(r_k,T_k)\right)^2
 +\sum_{k=1}^{N}M_A(r_k,T_k).
\end{equation}
\end{prop}

The proof of Proposition~\ref{prop:vitali-QIA} is given in
Sections~\ref{sec:mixed} and \ref{sec:modules}.  Assuming
Proposition~\ref{prop:vitali-QIA}, we now prove the divergence
part of Theorem~\ref{thm:main}.

\begin{proof}[Proof of the divergent case of
Theorem~\ref{thm:main}]
Suppose that
\(\sum_kM_A(r_k,T_k)=\infty\).  By
Lemma~\ref{lem:cumulative-mass},
\(\sum_k\mathcal L^m(E_k)=\infty\).  Applying
Lemma~\ref{lem:generalized-BC} to \((E_k)_{k\ge1}\) and using
Lemma~\ref{lem:cumulative-mass} together with \eqref{eq:vitali-QIA},
we obtain
\[
 \frac{\bigl(\sum_{k=1}^{N}\mathcal L^m(E_k)\bigr)^2}
      {\sum_{k,\ell=1}^{N}\mathcal L^m(E_k\cap E_\ell)}
 \gg
 \frac{\left(\sum_{k=1}^{N}M_A(r_k,T_k)\right)^2}
 {\left(\sum_{k=1}^{N}M_A(r_k,T_k)\right)^2
  +\sum_{k=1}^{N}M_A(r_k,T_k)}
\]
for all sufficiently large \(N\).
Since \(\sum_{k=1}^{N}M_A(r_k,T_k)\to\infty\), it follows that
\begin{equation}\label{eq:positive-vitali-limsup}
 \mathcal L^m(\limsup_{k\to\infty}E_k)>0.
\end{equation}

There is one minor distinction from the classical limsup decomposition.
Let \(\bb\in\limsup_kE_k\).  For every \(k\) such that \(\bb\in E_k\),
choose \(B_k\in\mathcal G_A(r_k,T_k)\) containing \(\bb\), and choose
\(\bq_k\in[-T_k,T_k]^n\cap\Z^n\) such that
\[
 B_k=B(\langle A\bq_k\rangle,r_k).
\]
If the vectors \(\bq_k\) took only finitely many values, one of them
would occur for infinitely many \(k\).  Since \(r_k\to0\), this would
force
\[
 \bb\in\set{\langle A\bq\rangle:\bq\in\Z^n}.
\]
Hence, outside this countable set, infinitely many distinct nonzero
vectors \(\bq_k\) occur.  For each such vector, monotonicity gives
\[
 r_k=\psi(T_k)\le\psi(\norm{\bq_k}).
\]
Consequently
\[
 (\limsup_{k\to\infty}E_k)
 \setminus\set{\langle A\bq\rangle:\bq\in\Z^n}
 \subset W_A(\psi).
\]
Together with \eqref{eq:positive-vitali-limsup}, this gives
\(\mathcal L^m( W_A(\psi))>0\).  Proposition~\ref{prop:zero-one}
upgrades the conclusion to full measure.
\end{proof}

\section{Mixed-scale intersections}\label{sec:mixed}

In this section, we estimate the measure of the mixed-scale intersections
\(E_k\cap E_\ell\) for \(k<\ell\).  For simplicity of notation, in this
section and the next we write
\[
 \mathcal G_k:=\mathcal G_A(r_k,T_k),\qquad
 M_k:=M_A(r_k,T_k),
\]
and, for \(k<\ell\),
\[
 M_{k,\ell}:=M_A(r_k,T_\ell).
\]
We begin with the local
packing estimate that counts how many selected balls at a later scale
can meet a fixed ball at an earlier scale.

\begin{lem}\label{lem:localized-vitali}
For \(k<\ell\) and \(B\in\mathcal G_k\),
\[
 \#\set{B'\in\mathcal G_\ell:B\cap B'\ne\varnothing}
 \ll\frac{N_A(r_k,T_\ell)}{N_A(r_\ell,T_\ell)}.
\]
\end{lem}

\begin{proof}
Choose \(\ba\in[-T_k,T_k]^n\cap\Z^n\) such that
\[
 B=B(\langle A\ba\rangle,r_k).
\]
For every \(B'\in\mathcal G_\ell\) meeting \(B\), choose
\(\bq_{B'}\in[-T_\ell,T_\ell]^n\cap\Z^n\) such that
\[
 B'=B(\langle A\bq_{B'}\rangle,r_\ell).
\]
Since \(B\cap B'\ne\varnothing\) and \(r_\ell\le r_k\), there is
\(\bp_{B'}\in\Z^m\) such that
\[
 \norm{A(\bq_{B'}-\ba)-\bp_{B'}}
 <r_k+r_\ell\le2r_k.
\]
Note that
\[
 \binom{A(\bq_{B'}-\ba)-\bp_{B'}}{\bq_{B'}-\ba}
 =u_A\binom{-\bp_{B'}}{\bq_{B'}-\ba}
 \in\LL_A.
\]
Translation by this lattice point gives
\[
 \#\left(
 \LL_A\cap\left[
 \binom{A(\bq_{B'}-\ba)-\bp_{B'}}{\bq_{B'}-\ba}
 +\mathcal B(r_\ell/2,T_\ell/2)
 \right]\right)
 =N_A(r_\ell/2,T_\ell/2).
\]
The first-coordinate projection of the translated box is
\[
 A(\bq_{B'}-\ba)-\bp_{B'}+[-r_\ell/2,r_\ell/2]^m.
\]
After translation by \(A\ba\), its image in \(\T^m\) is contained in
\(B'\).  Hence these translated boxes are pairwise disjoint as \(B'\)
ranges over the balls in \(\mathcal G_\ell\) meeting \(B\).  Moreover,
since \(r_\ell\le r_k\) and
\(\norm{\bq_{B'}-\ba}\le T_\ell+T_k\le3T_\ell/2\),
\[
 \binom{A(\bq_{B'}-\ba)-\bp_{B'}}{\bq_{B'}-\ba}
 +\mathcal B(r_\ell/2,T_\ell/2)
 \subset\mathcal B(5r_k/2,2T_\ell).
\]
Consequently,
\[
\begin{aligned}
 &\bigsqcup_{\substack{B'\in\mathcal G_\ell\\B\cap B'\ne\varnothing}}
 \left(
 \LL_A\cap
 \left[
 \binom{A(\bq_{B'}-\ba)-\bp_{B'}}{\bq_{B'}-\ba}
 +\mathcal B(r_\ell/2,T_\ell/2)
 \right]\right)\\
 &\qquad\subset
 \LL_A\cap\mathcal B(5r_k/2,2T_\ell).
\end{aligned}
\]
Taking cardinalities gives
\[
 \#\set{B'\in\mathcal G_\ell:B\cap B'\ne\varnothing}
 N_A(r_\ell/2,T_\ell/2)
 \le N_A(5r_k/2,2T_\ell).
\]
The desired bound follows from Lemma~\ref{lem:dilation}.
\end{proof}

The transition count gives the following mixed intersection estimate.

\begin{lem}\label{lem:mixed-packet}
For all sufficiently large \(k<\ell\),
\begin{equation*} \mathcal L^m(E_k\cap E_\ell)
 \ll\frac{M_kM_\ell}{M_{k,\ell}}.
\end{equation*}
\end{lem}

\begin{proof}
Since \(r_k\to0\), we have \(r_\ell\le r_k<1/4\) for all sufficiently
large \(k<\ell\), as required by Lemma~\ref{lem:cumulative-mass}.
By Lemmata~\ref{lem:cumulative-mass} and \ref{lem:localized-vitali}, we have
\[
\begin{aligned}
 \mathcal L^m(E_k\cap E_\ell)
 &\le\sum_{B\in\mathcal G_k}
       \sum_{\substack{B'\in\mathcal G_\ell\\B\cap B'\ne\varnothing}}
       \mathcal L^m(B\cap B')\\
 &\ll
 \frac{T_k^n}{N_A(r_k,T_k)}\,
 \frac{N_A(r_k,T_\ell)}{N_A(r_\ell,T_\ell)}\,r_\ell^m\\
 &=\frac{M_kM_\ell}{M_{k,\ell}}.
\end{aligned}
\]
\end{proof}

\section{Quasi-independence on average}\label{sec:modules}

We now estimate the double sum in
Proposition~\ref{prop:vitali-QIA}.  For \(N\ge1\), put
\[
 S_N:=\sum_{k=1}^{N}M_k,\qquad
 K_N:=\sum_{1\le k<\ell\le N}
      \frac{M_kM_\ell}{M_{k,\ell}}.
\]
By Lemma~\ref{lem:cumulative-mass}, the diagonal terms contribute
\(O(S_N)\).  For each exceptional initial index \(i\) and all
sufficiently large \(\ell\),
\[
 \mathcal L^m(E_i\cap E_\ell)
 \le\mathcal L^m(E_\ell)\ll M_\ell.
\]
Thus these intersections also contribute \(O(S_N)\); the finite initial
block is absorbed using \(S_N\ge M_1>0\).  Applying
Lemma~\ref{lem:mixed-packet} to the remaining off-diagonal terms gives
\begin{equation}\label{eq:second-moment-reduction}
 \sum_{k,\ell=1}^{N}\mathcal L^m(E_k\cap E_\ell)
 \ll S_N+K_N.
\end{equation}
Consequently, it is enough to prove
\begin{equation}\label{eq:mixed-kernel-target}
 K_N\ll S_N^2+S_N.
\end{equation}

We prove \eqref{eq:mixed-kernel-target} by separating two geometrically
different contributions.  For \(k<\ell\), let
\[
 \Gamma_{k,\ell}
 :=\LL_A\cap
 \spanR\bigl(\LL_A\cap\mathcal B(r_k,T_\ell)\bigr).
\]
This is the primitive sublattice spanned by the lattice points visible
in the mixed box.  If \(\Gamma_{k,\ell}\) has full rank, then
\(M_{k,\ell}\asymp 1\), and the corresponding pairs contribute the
quadratic term \(O(S_N^2)\).  We shall prove that all pairs with
\(\Gamma_{k,\ell}\) proper contribute only the linear term
\(O(S_N)\).

The proper-sublattice estimate has two ingredients.  First, all pairs
associated with one fixed proper sublattice can be charged to the
one-scale terms at their endpoints.  Second, nestedness of the mixed
boxes implies that each endpoint is charged by only \(O(1)\) distinct
proper sublattices.  We begin with the exterior-algebra description of
\(M_{k,\ell}\), which identifies the full-rank contribution and controls
each fixed proper sublattice.

\subsection{Exterior-product estimates}
\label{subsec:exterior-mass}

Let \(\Gamma\subset\LL_A\) be a primitive sublattice of rank \(\sigma\).
Choose a \(\Z\)-basis \(\bv_1,\ldots,\bv_\sigma\) and let
\(\omega_\Gamma=\bv_1\wedge\cdots\wedge\bv_\sigma\).
For \(g\in\operatorname{GL}_d(\R)\), write
\[
 g\omega_\Gamma:=(g\bv_1)\wedge\cdots\wedge(g\bv_\sigma).
\]
Its Euclidean norm is the covolume of \(g\Gamma\).  We use the supremum
norm in the standard exterior basis, so norm equivalence gives
\[
 \operatorname{covol}(g\Gamma)\asymp\norm{g\omega_\Gamma}.
\]

Let \(\mathbf e_1,\ldots,\mathbf e_d\) be the standard basis of
\(\R^d\).  For \(I=\{i_1<\cdots<i_\sigma\}\subset\{1,\ldots,d\}\),
write
\[
 \mathbf e_I:=\mathbf e_{i_1}\wedge\cdots\wedge\mathbf e_{i_\sigma}.
\]
Expand
\[
 \omega_\Gamma=\sum_{|I|=\sigma}c_I\mathbf e_I
\]
and define
\[
 \omega_{\Gamma,p}
 :=\sum_{\substack{|I|=\sigma\\|I\cap\{1,\ldots,m\}|=p}}
 c_I\mathbf e_I.
\]
The indexing conditions require \(p\le m\) and \(\sigma-p\le n\).
With
\[
 p_-(\sigma)=\max\{0,\sigma-n\},\quad
 p_+(\sigma)=\min\{m,\sigma\},
\]
we have
\begin{equation*} \omega_\Gamma
 =\sum_{p=p_-(\sigma)}^{p_+(\sigma)}\omega_{\Gamma,p}.
\end{equation*}
Under a change of \(\Z\)-basis by
\(U\in\operatorname{GL}_\sigma(\Z)\),
\[
 \omega_{\Gamma,p}'=(\det U)\omega_{\Gamma,p},
 \qquad \det U=\pm1.
\]
Thus \(\norm{\omega_{\Gamma,p}}\) is independent of the chosen basis.
Define
\begin{equation}\label{eq:Phi}
 \Phi_\Gamma(r,T)
 :=\max_{p_-(\sigma)\le p\le p_+(\sigma)}
 \norm{\omega_{\Gamma,p}}r^{m-p}T^{n-\sigma+p}.
\end{equation}
For \(\Gamma=\{0\}\), define
\(\Phi_\Gamma(r,T):=r^mT^n\).

For the box \(\mathcal B(r,T)\), put
\begin{equation*} \Gamma(r,T):=\LL_A\cap
 \spanR(\LL_A\cap\mathcal B(r,T)),\qquad
 \sigma(r,T):=\rank\Gamma(r,T).
\end{equation*}

The choice \(\Gamma(r,T)\) gives a two-sided estimate for \(M_A(r,T)\),
whereas an arbitrary \(\Gamma\) gives an upper bound.

\begin{lem}\label{lem:wedge-mass}
For the sublattice \(\Gamma(r,T)\),
\begin{equation}\label{eq:active-Phi}
 M_A(r,T)\asymp
 \Phi_{\Gamma(r,T)}(r,T).
\end{equation}
For every primitive rank-\(\sigma\) sublattice \(\Gamma\subset\LL_A\),
\begin{equation}\label{eq:arbitrary-Phi}
 M_A(r,T)\ll\Phi_\Gamma(r,T).
\end{equation}
When \(\sigma(r,T)=d\), both sides of \eqref{eq:active-Phi} are comparable
to one.
\end{lem}

\begin{proof}
As in Lemma~\ref{lem:minima-mass}, the choices
\(R=(T^nr^m)^{1/d}\) and \(t=\frac{mn}{d}\log(T/r)\) give
\[
 r e^{t/m}=T e^{-t/n}=R,
 \qquad a_t\mathcal B(r,T)=R[-1,1]^d.
\]
We first relate \(\Phi_\Gamma(r,T)\) to the covolume of
\(a_t\Gamma\).  The map \(a_t\) scales the first \(m\) standard basis
vectors by \(R/r\) and the remaining \(n\) by \(R/T\).  Hence
\[
 a_t\omega_{\Gamma,p}
 =\left(\frac Rr\right)^p
  \left(\frac RT\right)^{\sigma-p}\omega_{\Gamma,p}.
\]
Taking the supremum norm componentwise and multiplying by
\(R^{d-\sigma}\), we obtain
\begin{equation}\label{eq:flow-wedge}
 R^{d-\sigma}\norm{a_t\omega_\Gamma}
 \asymp
 \max_p\norm{\omega_{\Gamma,p}}r^{m-p}T^{n-\sigma+p}
 =\Phi_\Gamma(r,T).
\end{equation}

Let \(\lambda_1\le\cdots\le\lambda_d\) be the successive minima of
\(a_t\LL_A\), and let \(\mu_1\le\cdots\le\mu_\sigma\) be those of
\(a_t\Gamma\), all with respect to the supremum norm.
In \(\spanR(a_t\Gamma)\), the section of \([-1,1]^d\) lies between
the Euclidean balls of radii \(1\) and \(\sqrt d\).
Minkowski's second theorem and \eqref{eq:flow-wedge} therefore give
\[
 \Phi_\Gamma(r,T)
 \asymp R^{d-\sigma}\operatorname{covol}(a_t\Gamma)
 \asymp R^{d-\sigma}\prod_{i=1}^{\sigma}\mu_i.
\]
For \(\sigma=0\), the same formula holds with the empty product and
the covolume both equal to one.

We now prove \eqref{eq:arbitrary-Phi}.  The inclusion
\(a_t\Gamma\subset a_t\LL_A\) gives \(\lambda_i\le\mu_i\) for
\(1\le i\le\sigma\).  Hence Lemma~\ref{lem:minima-mass} yields
\begin{align*}
 M_A(r,T)
 &\asymp\prod_{i=1}^{d}\min\{R,\lambda_i\}
 \le R^{d-\sigma}\prod_{i=1}^{\sigma}\lambda_i\\
 &\le R^{d-\sigma}\prod_{i=1}^{\sigma}\mu_i
 \asymp\Phi_\Gamma(r,T).
\end{align*}

For \eqref{eq:active-Phi}, take \(\Gamma=\Gamma(r,T)\) and
\(\sigma=\sigma(r,T)\).  Since \(a_t\mathcal B(r,T)=R[-1,1]^d\),
\[
 a_t\Gamma
 =a_t\LL_A\cap\spanR\bigl(a_t\LL_A\cap R[-1,1]^d\bigr).
\]
Thus \(\sigma\) is the number of successive minima not exceeding \(R\):
\begin{equation*} \sigma=\#\{i\in\{1,\ldots,d\}:\lambda_i\le R\},
 \qquad \lambda_\sigma\le R<\lambda_{\sigma+1},
\end{equation*}
where \(\lambda_0=0\) and \(\lambda_{d+1}=\infty\).
Vectors realizing \(\lambda_1,\ldots,\lambda_\sigma\) lie in
\(a_t\LL_A\cap R[-1,1]^d\), and therefore belong to \(a_t\Gamma\).
This gives the reverse inequalities \(\mu_i\le\lambda_i\), so
\(\mu_i=\lambda_i\) for \(1\le i\le\sigma\).
Using these equalities in Lemma~\ref{lem:minima-mass}, we obtain
\[
 M_A(r,T)
 \asymp R^{d-\sigma}\prod_{i=1}^{\sigma}\lambda_i
 =R^{d-\sigma}\prod_{i=1}^{\sigma}\mu_i
 \asymp\Phi_\Gamma(r,T),
\]
which is \eqref{eq:active-Phi}.

Finally, if \(\sigma(r,T)=d\), then \(\Gamma(r,T)=\LL_A\).
Both \(a_t\) and \(\LL_A\) are unimodular, so
\(\operatorname{covol}(a_t\Gamma(r,T))=1\), proving the last assertion.
\end{proof}

\subsection{A fixed-sublattice estimate}
\label{subsec:fixed-sublattice}

We next bound the contribution from a fixed proper sublattice by
the sum over the indices appearing in its pairs.
For a finite set \(\mathcal R\) of pairs \((k,\ell)\) with \(k<\ell\),
define
\[
 \partial\mathcal R:=\bigcup_{(k,\ell)\in\mathcal R}\{k,\ell\}.
\]

\begin{lem}\label{lem:fixed-module}
Fix a primitive rank-\(\sigma\) sublattice \(\Gamma\subset\LL_A\), where
\(0\le \sigma<d\).  For every such \(\mathcal R\),
\begin{equation*} \sum_{(k,\ell)\in\mathcal R}
 \frac{M_kM_\ell}{\Phi_\Gamma(r_k,T_\ell)}
 \ll\sum_{j\in\partial\mathcal R}M_j.
\end{equation*}
The implied constant is uniform in \(\Gamma\) and \(\mathcal R\).
\end{lem}

\begin{proof}
For each \(j\), write \(\Phi_j=\Phi_\Gamma(r_j,T_j)\) and choose an
index \(p_j\) attaining the maximum in \eqref{eq:Phi} at
\((r_j,T_j)\), so that
\[
 \Phi_j=\norm{\omega_{\Gamma,p_j}}r_j^{m-p_j}T_j^{n-\sigma+p_j}.
\]
Write
\[
 a_j=n-\sigma+p_j,\qquad b_j=m-p_j.
\]
Then \(a_j,b_j\in\Z_{\ge0}\) and \(a_j+b_j=d-\sigma>0\), while
Lemma~\ref{lem:wedge-mass} gives \(M_j\ll\Phi_j\).

For \((k,\ell)\in\mathcal R\), let \(h=\ell-k\).
Evaluating the maximum at \(p_k\) and \(p_\ell\), respectively, gives
\begin{equation}\label{eq:Phi-lower-bounds}
\begin{aligned}
 \Phi_\Gamma(r_k,T_\ell)
 &\ge \norm{\omega_{\Gamma,p_k}}r_k^{b_k}T_\ell^{a_k}
 =\Phi_k2^{a_kh},\\
 \Phi_\Gamma(r_k,T_\ell)
 &\ge\Phi_\ell\left(\frac{r_k}{r_\ell}\right)^{b_\ell}
 \ge\Phi_\ell\varrho^{-b_\ell h},
\end{aligned}
\end{equation}
where the last inequality uses \eqref{eq:radius-decay}.
Consequently,
\begin{equation*} \frac{M_kM_\ell}{\Phi_\Gamma(r_k,T_\ell)}
 \ll\min\{M_\ell2^{-a_kh},\,
             M_k\varrho^{b_\ell h}\}.
\end{equation*}

For fixed \(\ell\), sum the first bound over \(k\) with \(a_k>0\).
Since then \(a_k\ge1\),
\begin{equation}\label{eq:first-geometric-arm}
 \sum_{\substack{k:(k,\ell)\in\mathcal R\\a_k>0}}
 M_\ell2^{-a_k(\ell-k)}
 \le M_\ell\sum_{h\ge1}2^{-h}
 \ll M_\ell.
\end{equation}
Among the remaining pairs, fix \(k\) and sum the second bound over
\(\ell\) with \(b_\ell>0\):
\begin{equation}\label{eq:second-geometric-arm}
 \sum_{\substack{\ell:(k,\ell)\in\mathcal R\\a_k=0,\ b_\ell>0}}
 M_k\varrho^{b_\ell(\ell-k)}
 \le M_k\sum_{h\ge1}\varrho^h
 \ll M_k.
\end{equation}
It remains to consider
\[
 \mathcal R_0:=\set{(k,\ell)\in\mathcal R:a_k=b_\ell=0}
\]
when this set is nonempty.  For \((k,\ell)\in\mathcal R_0\), the
equalities \(a_k=b_\ell=0\) give \(p_k=\sigma-n\) and \(p_\ell=m\).
Since \(p_-(\sigma)\le p_j\le p_+(\sigma)\),
\[
 0\le p_-(\sigma)\le\sigma-n,\qquad
 m\le p_+(\sigma)\le\sigma.
\]
Thus \(\sigma\ge\max\{m,n\}\), and in fact
\(p_k=p_-(\sigma)\) and \(p_\ell=p_+(\sigma)\).
As \(\omega_\Gamma\ne0\), the maximum in \eqref{eq:Phi} is positive.
Hence the maximizing components
\(\omega_{\Gamma,\sigma-n}\) and \(\omega_{\Gamma,m}\) are nonzero.
Consequently,
\[
 \Phi_k=\norm{\omega_{\Gamma,\sigma-n}}r_k^{d-\sigma},\qquad
 \Phi_\ell=\norm{\omega_{\Gamma,m}}T_\ell^{d-\sigma}.
\]
Since \(d-\sigma>0\), \eqref{eq:radius-decay} gives geometric decay of
\(\Phi_k\) along the left endpoints of \(\mathcal R_0\), with ratio at
most \(\varrho^{d-\sigma}\), while \(T_\ell=2^\ell\) gives geometric
growth of \(\Phi_\ell\) along its right endpoints, with ratio at least
\(2^{d-\sigma}\).
With \(a_k=b_\ell=0\), \eqref{eq:Phi-lower-bounds} gives
\[
 \Phi_\Gamma(r_k,T_\ell)\ge\max\{\Phi_k,\Phi_\ell\}.
\]
First fix \(k\) and consider the pairs with \(\Phi_\ell\le\Phi_k\).
Using \(\Phi_\Gamma(r_k,T_\ell)\ge\Phi_k\) and then
\(M_\ell\ll\Phi_\ell\), we obtain
\begin{equation}\label{eq:first-residual-arm}
 \sum_{\substack{\ell:(k,\ell)\in\mathcal R_0\\\Phi_\ell\le\Phi_k}}
 \frac{M_kM_\ell}{\Phi_\Gamma(r_k,T_\ell)}
 \ll \frac{M_k}{\Phi_k}
       \sum_{\substack{\ell:(k,\ell)\in\mathcal R_0\\\Phi_\ell\le\Phi_k}}
       \Phi_\ell
 \ll M_k.
\end{equation}
The last bound follows by summing in decreasing \(\ell\)-order:
the values \(\Phi_\ell\) grow geometrically and do not exceed \(\Phi_k\).

For pairs with \(\Phi_k<\Phi_\ell\), fix \(\ell\).
Using \(\Phi_\Gamma(r_k,T_\ell)\ge\Phi_\ell\) and then
\(M_k\ll\Phi_k\) gives
\begin{equation}\label{eq:second-residual-arm}
 \sum_{\substack{k:(k,\ell)\in\mathcal R_0\\\Phi_k<\Phi_\ell}}
 \frac{M_kM_\ell}{\Phi_\Gamma(r_k,T_\ell)}
 \ll \frac{M_\ell}{\Phi_\ell}
       \sum_{\substack{k:(k,\ell)\in\mathcal R_0\\\Phi_k<\Phi_\ell}}
       \Phi_k
 \ll M_\ell.
\end{equation}
For the last bound, we sum in increasing \(k\)-order, as the values
\(\Phi_k\) decrease geometrically and are less than \(\Phi_\ell\).

These four cases partition \(\mathcal R\).  Summing
\eqref{eq:first-geometric-arm}, \eqref{eq:second-geometric-arm},
\eqref{eq:first-residual-arm}, and \eqref{eq:second-residual-arm}
over their respective fixed indices in \(\partial\mathcal R\) gives
\[
 \sum_{(k,\ell)\in\mathcal R}
 \frac{M_kM_\ell}{\Phi_\Gamma(r_k,T_\ell)}
 \ll\sum_{j\in\partial\mathcal R}M_j.\qedhere
\]
\end{proof}

\subsection{Proof of Proposition~\ref{prop:vitali-QIA}}
\label{subsec:proof-vitali-QIA}

The remaining point is that each index occurs for only boundedly many
proper sublattices.  Fix \(N\ge1\).  For each proper primitive
sublattice \(\Gamma\subsetneq\LL_A\), define
\[
 \mathcal R_\Gamma
 :=\set{(k,\ell):1\le k<\ell\le N,\
                         \Gamma_{k,\ell}=\Gamma}.
\]

\begin{lem}\label{lem:incidence}
For every \(1\le j\le N\),
\[
 \#\set{\Gamma:j\in\partial\mathcal R_\Gamma}\le2d.
\]
\end{lem}

\begin{proof}
If \(j\in\partial\mathcal R_\Gamma\), then
\(\Gamma=\Gamma_{j,\ell}\) for some \(\ell>j\), or
\(\Gamma=\Gamma_{k,j}\) for some \(k<j\).
For the first family, we have
\[
 \mathcal B(r_j,T_{\ell_1})\subset \mathcal B(r_j,T_{\ell_2}),
 \qquad
 \Gamma_{j,\ell_1}\subset\Gamma_{j,\ell_2}
 \qquad(j<\ell_1<\ell_2).
\]
For the second, monotonicity of \(r_k\) gives
\[
 \mathcal B(r_{k_2},T_j)\subset \mathcal B(r_{k_1},T_j),
 \qquad
 \Gamma_{k_2,j}\subset\Gamma_{k_1,j}
 \qquad(k_1<k_2<j).
\]
In either chain, equal ranks force equal spans and hence equal
sublattices, by the definition of \(\Gamma_{k,\ell}\).
The proper sublattices have ranks in \(\{0,\ldots,d-1\}\), so at most
\(d\) distinct ones occur in each chain, giving at most \(2d\) in total.
\end{proof}

\begin{proof}[Proof of Proposition~\ref{prop:vitali-QIA}]
We prove \eqref{eq:mixed-kernel-target}, which implies the proposition
by \eqref{eq:second-moment-reduction}.

For a full-rank pair, \(\Gamma_{k,\ell}=\LL_A\), and
Lemma~\ref{lem:wedge-mass} gives \(M_{k,\ell}\asymp 1\).  Hence all
full-rank pairs together contribute
\[
 \sum_{\substack{1\le k<\ell\le N\\
                  \Gamma_{k,\ell}=\LL_A}}
 \frac{M_kM_\ell}{M_{k,\ell}}
 \ll\sum_{1\le k<\ell\le N}M_kM_\ell
 \le S_N^2.
\]

For each \(\Gamma\) with \(\mathcal R_\Gamma\ne\varnothing\),
Lemmata~\ref{lem:wedge-mass} and
\ref{lem:fixed-module} give
\begin{align*}
 \sum_{(k,\ell)\in\mathcal R_\Gamma}
 \frac{M_kM_\ell}{M_{k,\ell}}
 &\asymp\sum_{(k,\ell)\in\mathcal R_\Gamma}
 \frac{M_kM_\ell}{\Phi_\Gamma(r_k,T_\ell)}\\*
 &\ll\sum_{j\in\partial\mathcal R_\Gamma}M_j.
\end{align*}
Summing over these \(\Gamma\) and applying Lemma~\ref{lem:incidence} gives
\begin{align*}
 \sum_{\substack{1\le k<\ell\le N\\\Gamma_{k,\ell}\ne\LL_A}}
 \frac{M_kM_\ell}{M_{k,\ell}}
 &\ll\sum_\Gamma\sum_{j\in\partial\mathcal R_\Gamma}M_j\\*
 &=\sum_{j=1}^N M_j\,\#\set{\Gamma:j\in\partial\mathcal R_\Gamma}
 \le2dS_N.
\end{align*}
Together with the full-rank estimate, this proves
\eqref{eq:mixed-kernel-target}.
\end{proof}

\section{Concluding remarks and open problems}\label{sec:conclusion}

We close with three directions suggested by the proof: criteria beyond
dyadic regularity, exact Hausdorff dimension, and a weighted extension.

\subsection{The role of dyadic regularity}

The convergence argument in Section~\ref{subsec:convergence-case}
requires only that \(\psi\) be non-increasing and tend to zero.
For the divergence argument, the essential use of \Dcond{} occurs in
Lemma~\ref{lem:fixed-module}: the decay estimate
\eqref{eq:radius-decay} supplies the geometric tails in
\eqref{eq:second-geometric-arm} and \eqref{eq:second-residual-arm}.
These estimates control the contribution of each proper primitive
sublattice and yield the quasi-independence bound in
Proposition~\ref{prop:vitali-QIA}.  The same geometric summability is
used in Proposition~\ref{prop:scalar} to sum the dyadic scales within
a continued-fraction block.

Without \Dcond{}, the values \(\psi(2^k)\) may remain constant over
arbitrarily many consecutive scales, so these geometric tail estimates
are unavailable in general.

\begin{question}
For arbitrary \(A\) and positive non-increasing \(\psi\to0\), can one
obtain an explicit convergence--divergence criterion for
\(\mathcal L^m( W_A(\psi))\) from lattice counts over suitable
denominator blocks, generalizing the scalar Fuchs--Kim criterion?
\end{question}

\subsection{Hausdorff measure and dimension}
\label{subsec:hausdorff-gap}

For \(0<s<m\), the convergence and divergence conditions in
Corollary~\ref{cor:hausdorff} involve the two series
\[
 \sum_k\frac{2^{kn}\psi(2^k)^s}{N_A(\psi(2^k),2^k)}
 \qquad\text{and}\qquad
 \sum_k\frac{2^{kn}\psi(2^k)^s}{N_A(\psi(2^k)^{s/m},2^k)},
\]
respectively.  The first comes from covering at the original radius,
while the second comes from mass transference at an enlarged radius.
Their convergence and divergence conditions need not be complementary,
so the corollary does not give a zero--infinity law governed by a single
series.

In the one-dimensional case, Kim--Rams--Wang \cite[Theorem~1.1]{KRW18}
determine the exact Hausdorff dimension for every irrational rotation
and every positive non-increasing \(\psi\to0\), without \Dcond{}.
Their formula uses a series obtained by optimizing the covering cost
within each continued-fraction block.  This provides a model for an
exact dimension formula in higher dimensions.

\begin{question}
Suppose that \(A\Z^n\) is dense modulo \(\Z^m\).
Can one determine \(\dim_H W_A(\psi)\) for positive non-increasing
\(\psi\) satisfying \Dcond{} by optimizing covering radii and denominator
blocks?  In particular, what is the dimension for
\(\psi_v(q)=q^{-v}\), \(v>0\), when \(\sum_k\Delta_k(A)<\infty\)?
\end{question}

\subsection{The weighted setting}\label{sec:weighted}

Let
\[
 \mathbf r=(r_1,\ldots,r_m),\qquad
 \mathbf s=(s_1,\ldots,s_n)
\]
have positive coordinates normalized by
\[
 \sum_{i=1}^m r_i=m,\qquad \sum_{j=1}^n s_j=n,
\]
and define
\begin{align*}
 \norm{\bx}_{\mathbf r}
 &:=\max_{1\le i\le m}|x_i|^{1/r_i},
 &
 \norm{\bq}_{\mathbf s}
 &:=\max_{1\le j\le n}|q_j|^{1/s_j},\\
 \norm{\bx}_{\mathbf r,\Z^m}
 &:=\min_{\bp\in\Z^m}\norm{\bx-\bp}_{\mathbf r}.
\end{align*}
The weighted limsup set is
\[
  W_A^{\mathbf r,\mathbf s}(\psi)
 :=
 \left\{\bb\in\T^m:
 \begin{array}{l}
 \norm{A\bq-\bb}_{\mathbf r,\Z^m}
 <\psi(\norm{\bq}_{\mathbf s})\\
 \text{for infinitely many }\mathbf0\ne\bq\in\Z^n
 \end{array}
 \right\}.
\]
The corresponding clustering count and normalized mass are
\begin{align*}
 N_A^{\mathbf r,\mathbf s}(u,T)
 &:=
 \#\left\{(\bp,\bq)\in\Z^m\times\Z^n:
 \begin{array}{ll}
 |(A\bq-\bp)_i|\le u^{r_i}&(1\le i\le m),\\
 |q_j|\le T^{s_j}&(1\le j\le n)
 \end{array}\right\},\\
 M_A^{\mathbf r,\mathbf s}(u,T)
 &:=
 \frac{u^mT^n}{N_A^{\mathbf r,\mathbf s}(u,T)}.
\end{align*}
The normalization of the weights makes \(u^mT^n\) the volume scale and
recovers the present definitions when all weights equal \(1\).

Hussain and Ward proved a sufficient full-measure theorem in the weighted
nonsingular case \cite[Theorem~1]{HW24}.  We conjecture the following
weighted analogue of Theorem~\ref{thm:main} for an arbitrary fixed matrix.

\begin{conjecture}
Let \(A\in\Mat_{m,n}\), and let \(\psi:\R_+\to\R_+\) be
non-increasing and satisfy \Dcond{}.  For every pair of positive
weight vectors normalized as above,
\begin{equation}\label{eq:weighted-candidate-law}
 \mathcal L^m\bigl( W_A^{\mathbf r,\mathbf s}(\psi)\bigr)
 =
 \begin{cases}
 0&\text{if }\displaystyle
 \sum_k M_A^{\mathbf r,\mathbf s}(\psi(2^k),2^k)<\infty,\\[5pt]
 1&\text{if }\displaystyle
 \sum_k M_A^{\mathbf r,\mathbf s}(\psi(2^k),2^k)=\infty.
 \end{cases}
\end{equation}
\end{conjecture}

We expect that \eqref{eq:weighted-candidate-law} can be proved by
adapting the argument of Part~II: using rectangles with the prescribed
side lengths in the covering estimates and retaining the individual
coordinate weights in the exterior-product estimates and geometric sums.

\providecommand{\bysame}{\leavevmode\hbox to3em{\hrulefill}\thinspace}
\providecommand{\MR}{\relax\ifhmode\unskip\space\fi MR }
\providecommand{\MRhref}[2]{%
  \href{http://www.ams.org/mathscinet-getitem?mr=#1}{#2}
}
\providecommand{\href}[2]{#2}

\end{document}